\documentclass[11pt,a4paper]{amsart}
\usepackage{amscd}
\usepackage{amsmath,amsthm,amsfonts,graphicx,amssymb}
\usepackage[initials, alphabetic]{amsrefs}
\usepackage{hyperref}

\long\def\comment#1\endcomment{}
\usepackage{color}

\title{Actions of acylindrically hyperbolic groups on $\ell^1$}
\author{Cornelia Dru\c{t}u}
\email{cornelia.drutu@maths.ox.ac.uk}
\address{Mathematical Institute \\
 University of Oxford \\ Oxford, UK.}
\author{John M. Mackay}
\email{john.mackay@bristol.ac.uk}
\address{School of Mathematics \\ University of Bristol \\ Bristol, UK}
\thanks{The authors are grateful to Max Planck Institute for Mathematics in Bonn for its hospitality while work at this paper has been carried out.  
}
\date{\today}

\definecolor{darkgreen}{cmyk}{1,0,1,.2}
\allowdisplaybreaks[1] 

\DefineSimpleKey{bib}{primaryclass}{}
\DefineSimpleKey{bib}{archiveprefix}{}

\BibSpec{arXiv}{%
  +{}{\PrintAuthors}{author}
  +{,}{ \textit}{title}
  +{}{ \parenthesize}{date}
  +{,}{ arXiv:}{eprint}
  +{.}{}{transition}
}

\long\def\comment#1\endcomment{}

\long\def\green#1{\textcolor{darkgreen}{#1}}

\numberwithin{equation}{section}
\newtheorem{theorem}[equation]{Theorem}
\newtheorem{proposition}[equation]{Proposition}
\newtheorem{corollary}[equation]{Corollary}
\newtheorem{lemma}[equation]{Lemma}

\newtheorem{notation}[equation]{Notation}

\newtheorem{example}[equation]{Example}

\newtheorem{question}[equation]{Question}

\newtheorem{conjecture}[equation]{Conjecture}
\newtheorem{definition}[equation]{Definition}
\newtheorem{remark}[equation]{Remark}

\newtheorem{assumption}[equation]{Assumption}
\newtheorem{assumptions}[equation]{Assumptions}

\newtheoremstyle{citing}
  {3pt}
  {3pt}
  {\itshape}
  {}
  {\bfseries}
  {}
  {.5em}
  {\thmnote{#3}}

\theoremstyle{citing}
\newtheorem*{varthm}{}

\DeclareMathOperator{\diam}{diam}

\DeclareMathOperator{\Span}{Span}

\newcommand{\branch}{\mathfrak{B}}
\newcommand{\acts}{\curvearrowright}

\newcommand{\subseg}{\stackrel{\circ}{\subset}}

\newcommand{\mcgs}{\mathcal{MCG}(\Sigma)}

\newcommand{\cB}{\mathcal{B}}
\newcommand{\cA}{\mathcal{A}}
\newcommand{\cC}{\mathcal{C}}

\newcommand{\cM}{\mathcal{M}}

\newcommand{\R}{\mathbb{R}}

\newcommand{\N}{\mathbb{N}}
\newcommand{\Z}{\mathbb{Z}}

\newcommand{\id}{1} 

\newcommand{\tsh}[1]{\left\{\kern-.9ex\left\{#1\right\}\kern-.9ex\right\}}
\newcommand{\Tsh}[2]{\tsh{#2}_{#1}}

\newcommand{\mtx}[1]{\begin{pmatrix} #1 \end{pmatrix}}

\numberwithin{equation}{section}

\begin{document}

\begin{abstract}
	We construct explicit affine uniformly Lipschitz actions on $\ell^1$ for certain groups with hyperbolic features. For acylindrically hyperbolic groups, our actions have unbounded orbits, while for residually finite hyperbolic groups and for mapping class groups, the actions have proper orbits, with the induced $\ell^1$-metric quasi-isometric (respectively, almost quasi-isometric, up to a logarithmic factor) to the word metric. This allows us to construct proper continuous affine uniformly Lipschitz actions on $L^1$, with undistorted orbits, for all simple groups of real rank one.
\end{abstract}



\maketitle

\section{Introduction}
\label{sec:intro}

A countable group has \emph{Kazhdan's Property (T)} if and only if it has property $FH$, sometimes also denoted $FL^2$: every affine isometric action on an $L^2$ space has a fixed point.  
According to \cite{bader-gelander-monod-l1-fix}, Property (T) for a countable group is also equivalent to $FL^1$: every affine isometric action on an $L^1$ space has a fixed point. At the other end of the spectrum, a group is said to be \textit{a-T-menable} (or to have the \textit{Haagerup Property}) if it acts properly by affine isometries on a Hilbert space (equivalently, on an $L^1$ space) \cite{CCJJV-Haagerup-book,CDH-T}. 

In this paper, we consider more flexible actions, namely those that are ``affine uniformly Lipschitz'', that is, affine actions where the linear part is a uniformly bounded representation. For many groups with hyperbolic features -- including certain groups with Property (T) -- we construct explicit actions on $L^1$ with unbounded, even proper, orbits; moreover the proper actions we provide induce on the group left-invariant metrics quasi-isometric to word metrics, or almost quasi-isometric to them, up to a logarithmic factor. Prior to this paper there were no known examples of Property (T) groups
admitting an affine uniformly Lipschitz action on $L^1$ with proper (or
even unbounded) orbits.

These results also relate to the following conjecture. 
\begin{conjecture}[Y. Shalom, see \cite{nowak15}, Conjecture 35]
Every hyperbolic group acts properly by affine uniformly Lipschitz actions on a Hilbert space. 
\end{conjecture}

Shalom's conjecture is in line with the general belief that when property (T) is strengthened, rank one lattices (and, more generally, hyperbolic groups) fail to have it (and eventually become a-T-menable in a weaker sense), while higher rank lattices tend to have the stronger property. Indeed, the strengthening of Property (T) that consists in requiring that every affine uniformly Lipschitz action on a Hilbert space has a fixed point is satisfied by all the higher rank lattices, due to recent work of Oppenheim \cite{Oppenheim-higher rk}, and de Laat--de la Salle~\cite{de Laat-de la Salle-higher rk}.  

Many hyperbolic groups are known to be a-T-menable (e.g. discrete subgroups of $\mathrm{SO}(n,1)$ and $\mathrm{SU}(n,1)$, random groups with few relations), but of those which have Kazhdan's Property (T), only lattices of $\mathrm{Sp}(n,1)$ are known to satisfy Shalom's conjecture, due to work of Nishikawa~\cite{nishikawa-spn1}. I.\ Vergara also has recent work building actions on subspaces of $L^1$ \cite{vergara2}.

\subsection{Bounded and unbounded orbits}
In what follows, $G$ denotes a countable group, $E$ a normed vector space, $L(E,E)$ the algebra of linear maps from $E$ to $E$ and $\cB (E )$ the sub-algebra of bounded operators. Given a linear representation $\pi :G \to L(E,E)$, a \emph{cocycle} for $\pi$ is a function $\alpha:G \to E$ such that for all $g,h\in G$,
\[
	\alpha(gh) = \alpha(g)+\pi(g)\alpha(h).
\]
Equivalently, $g \cdot x = \pi(g)x+\alpha(g)$ defines an \emph{affine action} of $G$ on $E$. By a theorem of Mazur--Ulam, an isometric action on a Banach space is always affine, with $\pi$ an orthogonal representation.

Observe that an affine action of $G$ on $E$ is \emph{uniformly $L$-Lipschitz} if and only if its linear part is a representation of $G$ \emph{uniformly bounded by $L$}, that is:
\[ \sup_{g \in G} \| \pi(g)\|_{op} = \sup_{g\in G} \sup_{x\in E: \|x\|\leq 1} \| \pi(g)x \| \leq L.
\]
Following Bader--Furman--Gelander--Monod~\cite{bfgm}*{Definition 1.5}, given a Banach space $B$, we say that a group $G$ \emph{has property ($\bar{F}_B$)} if every (continuous) affine uniformly Lipschitz action on $B$ has bounded orbits. Actually, in~\cite{bfgm}, the definition requires fixed points rather than bounded orbits for ($\bar{F}_B$). While the two definitions are not equivalent for a single fixed space $B$, the property that a group satisfies ($\bar{F}_B$) for the entire class of superreflexive spaces is equivalent in the two formulations. On the other hand, any countable group $G$ acts by isometries on $\ell^1 (G)$, therefore on the subspace $\{ f\in \ell^1 (G) \; ;\; \sum_{g\in G} f(g) =1 \}$. The latter subspace has a renorming that makes it an $\ell^1$ space, on which $G$ therefore has a affine uniformly Lipschitz action with bounded orbits, but no global fixed point. Thus, for the (non-superreflexive) spaces $\ell^1$ and $L^1=L^1([0,1])$, our definition of ($\bar{F}_B$) in terms of bounded orbits (rather than global fixed points) is appropriate\footnote{
The property that every affine isometric action on $L^1$ has bounded orbits is equivalent with the stronger property requiring a global fixed point for every such action, by Bader--Gelander--Monod \cite[Theorem A]{bader-gelander-monod-l1-fix}.  Their method however does not extend to affine uniformly Lipschitz actions.}. 

It is easy to observe that ($\bar{F}_{L^1}$) implies ($\bar{F}_{\ell^1}$), see Proposition~\ref{prop:l1-to-L1}.

Our first main result is as follows.
\begin{theorem}\label{thm:acyl-hyp-unbounded-l1-action}
Every acylindrically hyperbolic group $G$ does not have property ($\bar{F}_{\ell^1}$), more precisely, for every $\epsilon>0$, $G$ admits an affine uniformly $(2+\epsilon)$-Lipschitz action with unbounded orbits on $\ell^1$, hence likewise on $L^1=L^1([0,1])$.
\end{theorem}
Recall that a countable group $G$ is \emph{acylindrically hyperbolic} if it admits an acylindrical, non-ele\-men\-tary action on a Gromov hyperbolic space; examples include non-elementary hyperbolic groups, mapping class groups that  are not  virtually abelian and groups of outer automorphisms of free non-abelian groups, ${\mathrm{Out}}(F_n),\, n\geq 2$.
Many of these groups have Kazhdan's Property (T), therefore, as mentioned above, do not admit affine \emph{isometric} actions on $L^1$ with unbounded orbits.

In Theorem \ref{thm:acyl-hyp-unbounded-l1-action}, we cannot replace $\ell^1$ by $\ell^p$, or even $L^p=L^p([0,1])$, for $p\in (1,\infty)$. 
The reason is the following example, which follows from work of Minasyan--Osin \cite{minasyan-osin} and de Laat--de la Salle \cite{deLaatdelaSalle}.
\begin{theorem}[Minasyan--Osin, de Laat--de la Salle, see \S\ref{ssec:MO-dLdlS}]\label{thm:MO-dLdlS-example}
There exists an a\-cyl\-in\-dri\-cal\-ly hyperbolic group $Q$ so that every affine uniformly Lipschitz action of $Q$ on a uniformly curved Banach space $X$ (in particular,\ $L^p$ or $\ell^p$, for $1<p<\infty$) has a fixed point, i.e.\ $Q$ has property ($\bar{F}_{X}$) for every uniformly curved Banach space $X$.
\end{theorem}

Uniformly curved Banach spaces have been introduced by G.~Pisier~\cite{pisier} as a subclass of the class of superreflexive Banach spaces, containing all the known examples of the latter type of Banach spaces. See \S\ref{ssec:MO-dLdlS} for details.

\subsection{Proper orbits} 
Another strengthening of Theorem~\ref{thm:acyl-hyp-unbounded-l1-action} that cannot be envisaged is to replace ``unbounded orbits'' by ``proper orbits'', as the example of Gromov monster groups that are acylindrically hyperbolic shows, see Example~\ref{ex:gromov-monster}.  

Nonetheless, we are able to construct explicit actions with proper orbits in some important cases. Moreover, we can quantify the properness of the action \textit{via} a distortion function $\rho $, as in the following definition.
\begin{definition}\label{def:rho}
	Let $G$ be a group with a proper left-invariant metric $d_G$ and with an action $G\acts X$ by uniformly Lipschitz transformations on a metric space $(X,d_X)$. Suppose $\rho:[0,\infty)\to\R$ is an increasing function with $\lim_{t\to\infty} \rho(t)=\infty$. We say that the action $G\acts X$ has \emph{$\rho$-proper orbits} if for some (any) $o \in X$ there exists $C>0$ so that 
\begin{equation}\label{eq:defrho}
d_X(o, g\cdot o) \geq \frac{1}{C}\rho(d_G(\id,g))-C,\, \forall g\in G.
\end{equation}
\end{definition}

For any such action, any function $\rho $ measuring the properness of the action must satisfy $\rho (t) \leq at+b, \forall t\geq 0$, for some positive constants~$a$ and $b$ (see Proposition~\ref{prop:upperrho}). In particular, for such a group $(G, d_G)$ an action $G\acts X$ has $t$-proper orbits if and only if every orbit map is a quasi-isometric embedding. In this case, we say that the orbits are \emph{undistorted}.

\medskip

Our next result, finding proper orbits for residually finite hyperbolic groups, is as follows. 

\begin{theorem}\label{thm:proper-res-fin-hyp}
If $G$ is a residually finite hyperbolic group, then, for every $\epsilon>0$, $G$ admits an affine uniformly $(2+\epsilon)$-Lipschitz action on $\ell^1=\ell^1(\N)$ with undistorted orbits, and hence likewise on $L^1=L^1([0,1])$.
\end{theorem}
Examples of residually finite hyperbolic groups with Property (T) include uniform lattices in $\mathrm{Sp}(n,1)$ and in $F_{4(-20)}$. A famous open problem is whether every hyperbolic group is residually finite, in which case more examples would come from random groups with many relations. Lattices in $\mathrm{Sp}(n,1)$ are known to have proper affine uniformly Lipschitz actions on $L^2$, by work of Nishikawa~\cite{nishikawa-spn1}, but there is no known method allowing to deduce from this the $L^1$ case (unlike for affine isometric actions). Nonetheless, Theorem~\ref{thm:proper-res-fin-hyp} and an argument with induced representations show the following.
\begin{corollary}\label{cor:rank-one-proper}
If $G$ is a simple Lie group of real rank one (which, according to Cartan's classification \cite{tits-classif}, implies that $G$ is one of the groups $\mathrm{SO}_0(n,1), \mathrm{SU}(n,1), \mathrm{Sp(n,1)}$ with $n \geq 2$, or $F_{4(-20)}$), then $G$ admits a proper continuous affine uniformly $(2+\epsilon)$-Lipschitz action on $L^1$ with undistorted orbits, for every $\epsilon>0$.
\end{corollary}
This result is entirely new for $\mathrm{Sp(n,1)}$ and $F_{4(-20)}$, where no proper continuous affine uniformly Lipschitz action on $L^1$ was previously known. The groups $\mathrm{SO}_0(n,1)$ and $\mathrm{SU}(n,1)$ are known to be a-T-menable, and hence have a proper continuous affine isometric action on $L^1$.    

The group $\mathrm{SO}_0(n,1)$ does have an affine isometric action on $L^1$ with undistorted orbits, induced by the action on the $n$-dimensional hyperbolic space preserving a structure of measured walls (\cite{robertson}, see also \cite[Example 3.7]{CDH-T}). The group $\mathrm{SU}(n,1)$ cannot have such an action, therefore an affine uniformly Lipschitz action on $L^1$ with undistorted orbits is about the best one can obtain for this group, and our construction of such an action is new. 

Besides residually finite hyperbolic groups, we can also strengthen Theorem~\ref{thm:acyl-hyp-unbounded-l1-action} to give proper, nearly undistorted, orbits in the case of mapping class groups.  

\begin{theorem}\label{thm:proper-mcg}
The mapping class group, $\mcgs$, of a surface $\Sigma$
of finite type admits an affine uniformly $(2+\epsilon)$-Lipschitz action on $\ell^1$ (hence also on  $L^1=L^1([0,1])$) with $\rho$-proper orbits, where one can take any $\epsilon>0$ and $\rho(t) = t/(1+\log^{\circ k}_+ t)$ for any $k \in \N$, with  
	$\log^{\circ k}_+(t)$ equal to the $k$-fold composition of $\log$ when that is defined and positive, and equal to $0$ otherwise.
\end{theorem}

\subsection{Metric distorsion}

Further context for our results is given by the work on compression exponents. A way of measuring the metric distorsion of a proper action is \emph{via} its \emph{equivariant compression exponent} $\alpha_X^{\#}$, defined as the supremum of the exponents $\alpha$ such that the inequality \eqref{eq:defrho} is satisfied with $\rho (t)=t^\alpha $.

This parameter has been studied for actions $G\acts X$ by \emph{isometric} transformations, as a way of making quantitative the investigation of a-T-menability and its Banach space versions. Exact values of compression exponents, and especially of the best possible compression exponents for proper actions on a certain class of Banach spaces (e.g.  Hilbert or $L^1$ spaces) are known for for only a few groups so far.

It was proved in \cite{naor-peres-imrn08} that for a finitely generated group $G$ acting isometrically on a Banach space $X$ with modulus of smoothness of power type $p>1$ (e.g. $X=L^p$), $\alpha_X^{\#} \leq \frac{1}{\min (2, p) b(G)},$ where $b(G)$ is a
parameter measuring the speed of divergence of canonical simple random walks on $G$. When $G$ is non-amenable, $b(G)=1$. The Naor-Peres inequality seems to point out that \emph{the equivariant representations on a Banach space that are metrically the most accurate (i.e. with the induced Banach metric the closest to the word metric of the group) should be the ones provided by actions on an $L^1$ space.} The a-T-menability does not necessarily imply the existence of an action on an $L^1$ space with undistorted orbits. This is not true even for simple examples such as the Heisenberg group \cite{cheeger-kleiner}. Moreover, there exist finitely generated amenable groups such that $\alpha_{L^1}^{\#} =0$ \cite{austin}. It is nevertheless known for a number of a-T-menable groups that $\alpha_{L^1}^{\#} =1$. Interestingly, for the proper actions on $L^1$ spaces by affine uniformly Lipschitz transformations that we obtain, the compression exponent is $1$ (and we even obtain undistorsion for residually finite hyperbolic groups).     

\subsection{Methods and remarks}
We use a similar method in the proofs of Theorems~\ref{thm:acyl-hyp-unbounded-l1-action}, \ref{thm:proper-res-fin-hyp} and \ref{thm:proper-mcg}. Given a group $G$ and a representation $\pi:G \to L(E,E)$ of it on an $L^1$ space $E$, consider an action on $E \oplus E$ defined by:
\[
	g \cdot \mtx{x \\ y} = \mtx{\pi(g) & D(g) \\ 0 & \pi(g)}\mtx{x \\ y} + \mtx{\alpha(g) \\ \beta(g)}, \quad \mtx{x\\ y} \in E \oplus E,
	\]
for suitable $D:G \to L(E,E), \alpha:G \to E,\, \beta:G \to E$.
This is a well-defined affine uniformly Lipschitz action provided $\pi$ is uniformly bounded, $\beta$ is a cocycle for $\pi$, $D$ is a uniformly bounded derivation in the sense of \cite{pisier-01-similarity-book} and $\alpha$ satisfies 
\begin{equation}\label{eq:introDeriv}
	D(g)\beta(h) = \alpha(gh) -\alpha(g)-\pi(g)\alpha(h)
\end{equation}
for all $g, h \in G$.  
Our idea is that certain known constructions of quasi-cocycles $\alpha$ for groups acting on $\ell^1$ can be used together with \eqref{eq:introDeriv} to define suitable uniformly bounded derivations $D$ (see \S\ref{sec:construction}).
The properness or unboundedness of $\alpha$ then implies the same property for orbits of the resulting affine uniformly Lipschitz action on $\ell^1\oplus\ell^1=\ell^1$.

In Theorem~\ref{thm:acyl-hyp-unbounded-l1-action}, we use the Brooks-type quasi-cocycles constructed in \cite{BBF-bdd-cohom-v2} to build the action.
In Theorems~\ref{thm:proper-res-fin-hyp} and \ref{thm:proper-mcg}, we use a variation of these quasi-cocycles provided in \cite{BBF-bdd-cohom-v1}, and put several quasi-cocycles together to get a proper quasi-cocycle, using methods developed by Bestvina--Bromberg--Fujiwara to show that such groups have ``property (QT)'', i.e. they admit proper isometric actions on a finite product of quasi-trees \cite{BBF-QT}.

In reference to Shalom's conjecture, one could attempt to make a similar construction for a Hilbert space, and the algebraic side of our argument goes through, however the boundedness of the derivations that we define heavily depends on being in $L^1$.
 

The following remark on the Lipschitz constants of our actions was prompted by a question of I.\ Vergara.
\begin{remark}
\begin{enumerate}
\item In Theorems~\ref{thm:acyl-hyp-unbounded-l1-action}, \ref{thm:proper-res-fin-hyp}, \ref{thm:proper-mcg} and Corollary~\ref{cor:rank-one-proper}, the `2' appearing in the Lipschitz constant $(2+\epsilon)$ can be replaced with $d=d_{BM}(\ell^1,\ell^1_0)$, the Banach--Mazur distance between $\ell^1$ and its subspace $\ell^1_0$ of zero sum sequences, see Remark~\ref{rmk:lip-const-banach-mazur-dist}. Our methods cannot yield $L$-Lipschitz actions for $L < d_{BM}(\ell^1,\ell^1_0)$. The constant $2$ comes from our estimate $d_{BM}(\ell^1,\ell^1_0) \leq 2$.  As kindly explained to us by W.\ Johnson,  $d>1$ (see Lemma~\ref{lem:l10-l1}).  The precise value of $d \in (1,2]$ seems to be unknown.   
\item An ultralimit argument along the lines of Fisher--Margulis shows that for any given countable group $G$ with Kazhdan's Property (T) there exists $\epsilon=\epsilon(G)>0$ so that $G$ does not admit an affine uniformly $(1+\epsilon)$-Lipschitz action on an $L^1$ space with unbounded orbits (see Proposition~\ref{prop:propT-implies-no-1eps-Lip-actions}).  Therefore for a group with Kazhdan's Property (T), Theorems~\ref{thm:acyl-hyp-unbounded-l1-action} and \ref{thm:proper-res-fin-hyp} cannot hold with Lipschitz constants arbitrarily close to $1$.
\end{enumerate}	
\end{remark}

The following significant question remains open.
\begin{question}
Find an infinite finitely generated group with ($\bar{F}_{\ell^1}$), or even ($\bar{F}_{L^1}$). 
\end{question}


If one is willing to sacrifice finite generation, then there are examples.
\begin{remark}
	Any ``strongly bounded group'' has ($\bar{F}_{\ell^1}$), for example $S_\infty$, the group of permutations of $\N$~\cite{Bergman-06-bdd-sym-groups}.
	Recall that a group is ``strongly bounded'' (or has ``property (OB)'') if any (continuous) isometric action of the group on any metric space has bounded orbits; note that any countable group with this property must be finite \cite{Cor-06-strongly-bounded}*{Remark 2.5}.
	Any such group $G$ also has the property that any uniformly Lipschitz action of the group on any metric space $(X,d)$ also must have bounded orbits:
	the metric $d_G(x,y) := \sup_{g\in G} d(g\cdot x, g\cdot y)$ is comparable to $d$, and the same action is isometric with respect to $d_G$.	
\end{remark}

\medskip

\begin{remark}
Subsequent to the initial release of this paper, there have been several
further works building affine uniformly Lipschitz actions on (subspaces of) $L^1$ spaces
by different methods, see \cite{vergara1,Gartland-hyperbolic,Gartland-Vergara-Zhang} and references therein.

The interest of the results in the present paper, besides being the first to provide examples of Property (T) groups with proper affine uniformly Lipschitz actions on $L^1$ spaces, resides in the fact that the actions are constructed explicitly by a novel method, that they are on an entire $L^1$ space (equivalently, they are defined \textit{via} a proper cocycle for a linear representation) rather than on a subspace of an $L^1$ space (equivalently, defined \textit{via} a proper left-invariant kernel of a certain type), and that in many cases they have the best distortion possible.
\end{remark}

\subsection*{Outline of the paper}
In Section~\ref{sec:examples} we provide some examples complementing Theorem \ref{thm:acyl-hyp-unbounded-l1-action} and some preliminary remarks. In Section \ref{sec:comm} we explain how an affine uniformly Lipschitz action on an $L^p$ space of a finite index subgroup (respectively, a certain type of lattice) $\Lambda$ in a group $G$ induces a similar action of $G$. A construction of $\ell^1$ actions given quasi-cocycles is shown in Section~\ref{sec:construction}.
Theorem~\ref{thm:acyl-hyp-unbounded-l1-action} is proved in Section~\ref{sec:acylind} using the quasi-cocycles of \cites{BBF-bdd-cohom-v2, hull-osin}.

In Section~\ref{sec:quasi-trees-l1-bound} we use a variation of the quasi-cocycles of \cite{BBF-bdd-cohom-v1} to build proper actions on $\ell^1$.  In Section~\ref{sec:res-fin-hyp-MCG} we combine this with the tools of \cite{BBF-QT} to prove Theorems~\ref{thm:proper-res-fin-hyp} and \ref{thm:proper-mcg}.

\subsection*{Acknowledgements}
Both authors are grateful to the Max Planck Institute for Mathematics in Bonn for its hospitality during several research visits, while work at this paper has been carried out.

We thank I.\ Vergara for helpful comments prompting us to clarify the Lipschitz constants of our actions, W.\ Johnson for the lower bound in Lemma~\ref{lem:l10-l1}, and A.\ Sisto for useful comments.

\medskip

For the purpose of Open Access, the authors have applied a CC BY public copyright licence to any Author Accepted Manuscript version arising from this submission.

\section{Examples and preliminary remarks}\label{sec:examples}

We describe here two sets of examples illustrating that Theorem \ref{thm:acyl-hyp-unbounded-l1-action} is, in some sense, optimal. 
We also make observations on orbits of uniformly Lipschitz actions and on the relationship between actions on $\ell^1$ and $L^1$.

\subsection{An acylindrically hyperbolic group with many fixed point properties}\label{ssec:MO-dLdlS}

As explained in the introduction, in Theorem \ref{thm:acyl-hyp-unbounded-l1-action}, the $\ell^1$ space cannot be replaced with an $L^p$ space, for $p\in (1,\infty)$. Indeed, the combined work of Minasyan--Osin and de Laat--de la Salle yields acylindrically hyperbolic groups $Q$ so that any affine uniformly Lipschitz action of $Q$ on any uniformly curved Banach space $X$ (e.g.\ $L^p$ or $\ell^p$ for $1<p<\infty$) has a fixed point.
 
\begin{proof}[Proof of Theorem \ref{thm:MO-dLdlS-example}]
	By Minasyan--Osin \cite{minasyan-osin}*{Theorem 1.1}, every countable family of countable acylindrically hyperbolic groups has a common finitely generated acylindrically hyperbolic quotient. In particular, there exists a finitely generated infinite acylindrically hyperbolic group $Q$ that is quotient of all the non-elementary hyperbolic groups. 

	G.~Pisier introduced in \cite{pisier} a subclass of the class of super-reflexive Banach spaces, the \emph{uniformly curved} Banach spaces. He proved that all known examples of super-reflexive spaces are uniformly curved, and asked whether super-reflexive implies uniformly curved \cite[Problem 2.4]{pisier}, a question that remains open.

De Laat and de la Salle proved in \cite{deLaatdelaSalle} that random groups in the triangular model (which are non-elementary hyperbolic) satisfy fixed point properties for affine isometric actions on arbitrarily large classes of uniformly curved Banach spaces (see \cite{DruMac-lp-random} for the $L^p$ case). Therefore, the acylindrically hyperbolic group $Q$ defined above inherits the fixed point property for affine isometric actions on any uniformly curved Banach space.

	Consequently, $Q$ has the fixed point property for affine uniformly Lipschitz actions on uniformly curved Banach spaces. Indeed, let $Q\acts X$ be such an action, given by:
$$
x\mapsto \pi(g) x + \zeta(g).
$$

One can define an equivalent norm on $X$ by:
$$
\| x\|_Q = \sup_{g\in Q} \| \pi(g) x\|.
 $$ 
The Banach space $(X,\|\cdot\|_Q)$ is linearly isomorphic and $L$-bi-Lipschitz equivalent (or, with a different terminology, $L$-isomorphic) to a uniformly curved Banach space, hence $(X,\|\cdot\|_Q)$ is itself uniformly curved \cite{pisier}*{Page 14}. The action of $Q$ on $(X,\|\cdot\|_Q)$ is affine isometric, hence it has a fixed point.\end{proof}

\subsection{Graphical small cancellation groups, Gromov monsters, restriction to cocycles that are not proper}

We begin by noting that Theorem \ref{thm:acyl-hyp-unbounded-l1-action} has the following consequence. 

\begin{corollary}\label{cor:smallCancell}
Every infinitely presented graphical Gr(7) small cancellation
group and every cubical small cancellation group admits an affine uniformly Lipschitz action on $\ell^1$ with unbounded orbits.
\end{corollary}

This follows from the results of Gruber and Sisto~\cite{grubersisto} and of Arzhant\-se\-va--Hagen~\cite{arzhant-hagen-acyl-hyp-cubical}, showing that infinitely presented graphical Gr(7) small cancellation groups, respectively cubical small cancellation groups, are acylindrically hyperbolic. The former class of groups includes infinitely presented
classical $C(7)$-groups and, hence, classical $C'(\frac{1}{6} )$-groups.

Small cancellation theory is a technique allowing to build infinite finitely generated groups with unusual properties (the so-called ``infinite monsters''),  mostly used to produce counter-examples to various conjectures. 
It is more challenging to obtain positive results about the entire class of small cancellation groups, and while such results have been proven for algebraic and geometric properties, for analytic properties much less has been known, until recently.

Graphical small cancellation groups have been introduced by Gromov in \cite{gromovrw} with the view to construct infinite groups with prescribed embedded subgraphs in their Cayley graphs. Gromov used this technique combined with a probabilistic argument to build the so called {\emph{Gromov monsters}}, groups that contain families of expanders embedded in a weak sense into their Cayley graphs (see~\cite{arzdel} for further details). More recently, Osajda found Gromov monster groups which contain isometrically embedded expanders and which satisfy the graphical small cancellation property~\cite{Osajda-20-monsters}. Gromov monsters are the only known counterexamples to the Baum--Connes conjecture with coefficients~\cite{higsonlafforgueskandalis}. It is the same Gromov monsters that illustrate the fact that in Theorem \ref{thm:acyl-hyp-unbounded-l1-action}, `unbounded orbits' cannot be strengthened to `proper orbits'.     

\begin{example}\label{ex:gromov-monster}
	Let $G$ be a Gromov monster group as constructed by Osajda~\cite{Osajda-20-monsters}. As $G$ is acylindrically hyperbolic, it has an affine uniformly Lipschitz action on $\ell^1$ (or $L^1$) with unbounded orbits, by Theorem \ref{thm:acyl-hyp-unbounded-l1-action} (see Corollary~\ref{cor:smallCancell}).
	On the other hand, $G$ has isometrically embedded copies of expanders in its Cayley graphs, so cannot embed uniformly into any $L^p$ space, for $p\in (0,2]$~\cite{wells-williams}. 
	Therefore, the uniformly bounded representations on $\ell^1$ obtained by Theorem \ref{thm:acyl-hyp-unbounded-l1-action} do not admit proper cocycles in general, only unbounded ones, not even if $\ell^1$ is replaced with $L^1$.   
\end{example}

\subsection{Remarks on uniformly Lipschitz actions}

We justify the standard fact mentioned in the introduction that in building an action on $\ell^1$ we get an action on $L^1=L^1([0,1])$ as well.
\begin{proposition}\label{prop:l1-to-L1}
	If $G$ admits an affine uniformly Lipschitz action on $\ell^1$ with unbounded (respectively $\rho$-proper) orbits, then $G$ admits an affine uniformly Lipschitz action on $L^1([0,1])$ with unbounded (respectively $\rho$-proper) orbits.	
\end{proposition}
\begin{proof}
	We can view $\ell^1$ as a complemented subspace of $L^1$, i.e.\ $L^1 = Y \oplus Z$ is a topological direct sum of closed subspaces $Y, Z$ with $Y$ isometrically isomorphic to $\ell^1$, see for example Albiac--Kalton~\cite{AlbKal-Banach-space-book}*{Lemma 5.1.1}.
	Therefore, any affine action of a group $G$ on $\ell^1$ can be extended to one on $L^1$ via $g\cdot (y,z) = (g\cdot y, z)$.
	This action is uniformly Lipschitz if the action on $\ell^1$ is, and it clearly preserves the properties of having unbounded or $\rho$-proper orbits.
\end{proof}

Next, again as standard, we note that the properness function $\rho$ of a uniformly Lipschitz affine action grows at most linearly.
\begin{proposition}\label{prop:upperrho} 
	Suppose $G$ is a group endowed with a proper left-invariant metric $d_G$ and acting on a metric space $(X,d_X)$ by uniformly $L$-Lipschitz affine transformations.
	Then for every $o \in X$ there exists $L_o, C_o \geq 0$   so that
\begin{equation}\label{eq:linear-upper}
d_X(o, g\cdot o) \leq L_o \lceil d_G(\id_G,g) \rceil +C_o,\, \forall g\in G,
\end{equation}
	moreover, we can either take
	\[
		L_o = L \diam\{ g\cdot o : d_G(\id_G, g) \leq 1 \} \text{ and } C_o = 0,
	\]
	or fix a basepoint $x_0$ and take
	\[
		L_o = L \diam\{ g\cdot x_0: d_G(\id_G,g)\leq 1\} \text{ and } C_o=(L+1)d_X(o,x_0).
	\]
\end{proposition}
\begin{proof}
	For any $g \in G$ write $g= s_1s_2\dots s_k$, where $s_i, i\in \{1,\dots ,k\}$, are in the closed ball $B_G (\id_G,1)$ of radius $1$ around the identity element $\id_G\in G$, and $k \leq \lceil d_G(\id_G, g)\rceil$.

	For the first choice of constants, consider:
	\begin{align*}
		d_X(o, g\cdot o)
		& \leq d_X(o, s_1\cdot o) + d_X(s_1\cdot o, s_1s_2 \cdot o ) + \cdots 
		\\ & \quad + d_G((s_1\cdots s_{k-1}) \cdot o, (s_1\cdots s_k)\cdot o)
		\\ & \leq L k \sup\{ d_X(o, g \cdot o) : g \in B_G(\id_G,1)\}.
	\end{align*}
	For the second choice of constants, fix $x_0$ and consider:
	\begin{align*}
		d_X(o,g\cdot o) & \leq d_X(o, x_0)+d_X(x_0, g\cdot x_0)+d_X(g\cdot x_0, g\cdot o)
		\\ & \leq (1+L)d_X(o,x_0) + L k \sup\{d_X(x_0,g\cdot x_0): g \in B_G(\id_G,1)\}
	\end{align*}
	by the same argument as before.
\end{proof}

Finally, as mentioned in the introduction, we note the following rigidity result for countable groups with Property (T) acting on $L^1$.
\begin{proposition}
	\label{prop:propT-implies-no-1eps-Lip-actions}
	If $G$ is a countable group with Property (T), then there exists $\epsilon=\epsilon(G)>0$ so that $G$ does not admit an affine uniformly $(1+\epsilon)$-Lipschitz action on an $L^1$ space with unbounded orbits.
\end{proposition}
\begin{proof} We begin by noting that Property (T) implies that $G$ is finitely generated. Assume for a contradiction that $G$ has a sequence of actions $\rho_n$ with unbounded orbits by affine uniformly $(1+\epsilon_n)$-Lipschitz actions on an $L^1$ space, where $ \epsilon_n \to 0$. Let $S$ be a generating set. An argument as in the proof of \cite[Theorem 19.22]{DrutuKapovich}, together with \cite[Corollary 19.18]{DrutuKapovich} imply that a rescaled ultralimit of the actions $\rho_n$ is an action $\rho_\omega$ by affine isometries on an $L^1$--space $X_\omega$ such that for every $x\in X_\omega$ the diameter of $Sx$ is at least $1$. This contradicts the theorem in \cite{bader-gelander-monod-l1-fix}, stating that every such action of a group with Property (T) must have a global fixed point.    
\end{proof}

\section{Lattices and induced representations}\label{sec:comm}

To show that a group admits a proper (respectively, unbounded orbits) affine uniformly Lipschitz action on an $L^p$ space, by a fairly standard induction it suffices to show the same property for a finite index subgroup.  
We use this result in the proofs of Theorems~\ref{thm:proper-res-fin-hyp} and \ref{thm:proper-mcg}.
We state the results more generally, as it is interesting to see how these properties behave in the more general context of lattices in locally compact groups, and needed for Corollary~\ref{cor:rank-one-proper}.
The following propositions require some adaptations from standard arguments, as we have only uniformly Lipschitz actions rather than isometric actions.

First, for completeness, we note that if ($\bar{F}_B$) holds for a lattice then it holds for the ambient group; we do not use this result elsewhere.  
\begin{proposition}
	[cf.\ \cite{Bek-dlH-Val-Kazhdan-book}*{Proposition 2.5.5}, \cite{bfgm}*{Proposition 8.8}]
	Suppose $G$ is a locally compact group with Haar measure $\mu$, and $\Lambda$ a lattice in $G$.
	For any Banach space $B$, if $\Lambda$ has ($\bar{F}_B$) then $G$ has ($\bar{F}_B$).	
\end{proposition}
\begin{proof}
	Let $\cM(B)$ be the set of all closed and bounded subsets of $B$, which is a complete metric space when endowed with the Hausdorff metric $d_H$.
	Suppose $G$ has a continuous affine $L$-Lipschitz action on $B$, and hence also on $\cM(B)$.  
	The $\Lambda$-orbits in $B$ are bounded by ($\bar{F}_B$); let $O = \overline{\Lambda \cdot 0} \in \cM(B)$ be the closure of one such orbit.

	Let $\mu$ also denote the $G$-invariant measure on $G/\Lambda$, which we may assume satisfies $\mu(G/\Lambda)=1$.
	Consider the continuous map $\Phi:G/\Lambda \to \cM(B), g\Lambda \mapsto g \cdot O$.  Let $\nu = \Phi_*(\mu)$ be the push-forward of $\mu$, which is a $G$-invariant measure on $\cM(B)$.

	Since $\lim_{R\to \infty} \nu(B_{d_H}(O,R))=1$, there exists $R_0$ such that $\nu(B_{d_H}(O,R_0)) > 1/2$.
	Now for any $g \in G$, 
	$B_{d_H}(g\cdot O,L R_0) \supset g \cdot B_{d_H}(O,R_0)$, hence 
	$\nu(B_{d_H}(g\cdot O,L R_0))>1/2$ also.
	Therefore $B_{d_H}(g\cdot O, L R_0) \cap B_{d_H}(O,R_0) \neq \emptyset$ and consequently $d_H(O, g \cdot O) \leq (L+1)R_0$.

	In conclusion, for all $g \in G$, 
	$d(0,g \cdot 0) \leq \diam(O)+d_H(O, g\cdot O) \leq \diam(O)+(L+1)R_0$.  
\end{proof}

The existence of a proper action on a certain type of Banach space is obviously inherited by lattices.

\begin{proposition}
	Suppose $G$ is a locally compact group and $\Lambda < G$ is a lattice whose word metric $d_\Lambda$ satisfies $d_G(x,y) \geq \tau(d_\Lambda(x,y)), \forall x,y\in \Lambda$, where $\tau:[0,\infty)\to [0,\infty)$ is a proper function.
	If $G$ acts with $\rho$-proper orbits on a metric space $X$,
	then $\Lambda$ acts with $\rho \circ \tau$-proper orbits on $X$.
\end{proposition}

In order to show that unbounded/proper actions of a lattice induce similar actions of the ambient group, we require that the lattice be ``integrable'' in an appropriate sense.
\begin{definition}[cf.\ \cite{bfgm}*{Definition 8.2} ]
	\label{def:integrable-lattice}
	Let $G$ be a locally compact group admitting a proper left-invariant metric and a Haar measure $\mu$, and let $\Lambda$ be a finitely generated lattice of $G$ with a word norm $|\cdot |_\Lambda$ (i.e. $| h |_\Lambda = d_\Lambda (\id,h)$, where $d_\Lambda $ is a word metric on $\Lambda$).  
	We say $\Lambda$ is \emph{$p$-integrable} for some $p \in [1,\infty)$
	if there exists a Borel fundamental domain $B$ for $\Lambda$ acting on $G$, so that $\{hB\}_{h\in \Lambda}$ partitions $G$ and $\mu(\partial B)=0$, and for all $g\in G$,
	\[
		\int_B |\alpha(g,\gamma) |_\Lambda^p \,d\mu(\gamma) < \infty,
	\]
	where $\alpha(g,\gamma) \in \Lambda$ is uniquely defined by $\alpha(g,\gamma)\gamma g^{-1} \in B$ for $g \in G, \gamma \in B$.
\end{definition}
This is a mild restriction.  Every uniform lattice, in particular every finite index subgroup of a finitely generated group, is $p$-integrable for all $p\geq 1$. Lattices in all simple Lie groups of rank one (except for those locally isomorphic to $\mathrm{PSL}(2,\R)$) and irreducible lattices in semisimple Lie groups of rank at least two are $1$-integrable due to work of Shalom~\cite{Shalom-00-rigidity-comm-lattices}*{\S 2} (see \cite{Bad-Fur-Sauer-13-IME-rigidity}*{Theorem 1.9}).

We now state the main result of this section.
\begin{proposition}[cf.\ \cite{BBF-QT}*{\S 2.2}, \cite{Bek-dlH-Val-Kazhdan-book}*{Appendix E}]\label{prop:lattices}
	Suppose $G$ is a locally compact group admitting a proper left-invariant metric and a Haar measure $\mu$, and $\Lambda$ is a finitely generated $p$-integrable lattice in $G$. We likewise denote by $\mu$ the (right) $G$-invariant measure on $\Lambda\backslash G$ induced by the Haar measure. If $\Lambda$ admits an affine uniformly $L$-Lipschitz action on some $L^p(\Omega,\nu)$ then  
	$G$ admits a continuous affine uniformly $L$-Lipschitz action on $L^p(\Lambda\backslash G \times \Omega, \mu\times\nu)$ such that:

\begin{enumerate}

\item\label{item:unbded} if $\Lambda \acts L^p(\Omega,\nu)$ has unbounded orbits, then $G\acts L^p(\Lambda\backslash G \times \Omega, \mu \times\nu)$ has unbounded orbits;

\item\label{item:proper} if $\Lambda \acts L^p(\Omega,\nu)$ has $\rho$-proper orbits, then $G\acts L^p(\Lambda\backslash G \times \Omega, \mu \times\nu)$ has $\rho \left( \frac{1}{C}t-C \right)$-proper orbits, for some $C>0$.  
\end{enumerate}
\end{proposition}
Alternatively, Proposition~\ref{prop:lattices}, \eqref{item:unbded}, states that if a locally compact group $G$ has ($\bar{F}_{L^1}$) then so does any $1$-integrable lattice in it. 
Note too that if $\Lambda \leq G$ is a finite index subgroup and $\Omega = \N$, then $L^1(\Lambda\backslash G \times \Omega) = L^1(\N) = \ell^1$.
\begin{proof}
	The construction of the action proceeds in several stages, with the verification of \eqref{item:unbded} and \eqref{item:proper} at the end.

	\emph{The normed vector space.} \quad Fix a Borel fundamental domain $B$ as in Definition~\ref{def:integrable-lattice}. 
Let $X = L^p(\Omega,\nu)$, with its given $\Lambda$-action.
We now define a Banach space $Y$ which has two equivalent guises: one as the space $L^p(B,\mu; X)$ in the sense of Bochner integration, and the other as the quotient set of the space of functions
	\[
\left\{ f:G \to X : f(h\gamma )=h\cdot f(\gamma ),\ \forall h\in \Lambda, \gamma \in G,\; \int_B \|f(\gamma )\|_X^p \, d\mu(\gamma ) <\infty \right\} 
	\]
with respect to the equivalence relation \[ f_1\sim f_2\Leftrightarrow \int_B \|f_1(\gamma ) -f_2(\gamma )\|_X^p \, d\mu(\gamma )=0.\] With the second definition, the quotient space $Y$ is endowed with the norm
	\[
		\|f\|_Y = \left( \int_B \|f(\gamma )\|_X^p \, d\mu(\gamma ) \right)^{1/p}.
	\]

Let $0_Y \in Y$ be the function that is identically $0\in X$ on $B$ (and $h\cdot 0$ on $hB$, $\forall h\in \Lambda$).
	Define the addition and scalar multiplication on $Y$ by
	$(tf+f')(\gamma) = tf(\gamma)+f'(\gamma)$ for $t \in \R$, $f,f' \in Y$, $\gamma \in B$, and extend equivariantly by $\Lambda$.
	The set $Y$ becomes a normed vector space with origin $0_Y$. Both its norm and the two operations on it depend on the choice of the domain $B$. The space $Y$ is isometric to 
	\[
		L^p(B,\mu; X) = L^p(B,\mu; L^p(\Omega,\nu;\R))
	= L^p(B \times \Omega,\mu\times\nu;\R). \]
	
\medskip

\noindent \emph{The action of $G$ on $Y$. Cocycle for the action.}\quad We define a left action of $G$ on $Y$ by
	$g \cdot f(\gamma) = f(\gamma g^{-1})$ for all $g,\gamma \in G$. The condition of $\Lambda$-equivariance is clearly satisfied by $g \cdot f$. We will show later that $\|g\cdot f\|_Y < \infty$. Note that for $a,b,\gamma\in G$,
	\[
		((ab)\cdot f)(\gamma) = f(\gamma b^{-1} a^{-1}) = (a\cdot(b\cdot f))(\gamma).
	\]
	
\medskip

For any $g\in G$, we define a left action $\bullet:G \times B \to B$ and a cocycle $\alpha:G \times B \to \Lambda$ by
	\[ g \bullet \gamma  := \alpha(g,\gamma )\gamma g^{-1} \in B, \forall \gamma  \in B, g\in G. \]
	Every $g\in G$ defines a partition of $B$ as $\bigsqcup_{h\in \Lambda} B_{g,h}$, where $B_{g,h}$ denotes the intersection $B\cap hBg$, and $\gamma \mapsto g\bullet \gamma $ restricted to every $B_{g,h}$ is the composition between a right multiplication by $g^{-1}$ and a left multiplication by $h^{-1} = \alpha(g,\gamma)$. Hence, this defines a left measure-preserving action of $G$ on $B$. 
	In terms of elements of $Y$,
	for $\gamma \in B_{g,h}$, 
	\[
		(g \cdot f)(\gamma)
		= h \cdot f(h^{-1} \gamma g^{-1})
		= h \cdot f(g \bullet \gamma)
		= \alpha(g,\gamma)^{-1} \cdot f(g \bullet \gamma).
	\]

\medskip

\noindent \emph{The action on $Y$ is well-defined.}\quad As $\Lambda$ is a lattice, $\mu$ is also right-invariant, 
	thus for $f\in Y, g\in G$,
	\begin{equation}\label{eq:lattice-gf}
 			\|g\cdot f\|_Y^p
		 = \int_B \|\alpha(g,\gamma)^{-1}  \cdot f(g \bullet \gamma) \|_X^p \,d\mu(\gamma).
	\end{equation}
	Recall from Definition~\ref{def:integrable-lattice} that $\Lambda$ is finitely generated with word norm $|\cdot |_\Lambda$ associated to some symmetric generating set.
	Since the action of $\Lambda$ on $X$ is uniformly $L$-Lipschitz, according to Proposition~\ref{prop:upperrho}, we can write
$$
		\|h\cdot 0\|_X \leq |h|_\Lambda L C,\; \forall h\in \Lambda,
$$
	where $C = \max_{s\in S} \|s\cdot 0\|_X$.
	So for every $h \in \Lambda$ and $x\in X$, 
	\begin{align*}
		\|h\cdot x\|_X  &
		\leq \|h\cdot x - h\cdot 0\|_X+\|h\cdot 0\|_X 
		\\ & \leq L \|x\|_X+\|h\cdot 0\|_X
		\\ & \leq L \|x\|_X+ CL |h|_\Lambda .
	\end{align*}

The above and the inequality $(a+b)^p\leq 2^p (a^p +b^p), \forall a,b\geq 0$, imply that 
	\begin{align*}
		\|g\cdot f\|_Y^p
		& = \int_B \|\alpha(g,\gamma)^{-1}  \cdot f(g \bullet \gamma) \|_X^p \,d\mu(\gamma)
		\\ & \leq (2L)^p \int_B \|f(g \bullet \gamma)\|_X^p d\mu(\gamma) + (2CL)^p \int_B |\alpha(g,\gamma)^{-1}|_\Lambda^p d\mu(\gamma) 
		\\ & = (2L)^p \int_B \|f(\gamma)\|_X^p d\mu(\gamma) + (2CL)^p \int_B |\alpha(g,\gamma) |_\Lambda^p d\mu(\gamma) < \infty
	\end{align*}
	by $p$-integrability, where the last equality uses that the action $B \mapsto B, \gamma\mapsto g \bullet \gamma$ is measure preserving.

Therefore $f\in Y$ implies $g\cdot f \in Y$.

\medskip

\noindent \emph{The action on $Y$ is affine.}\quad Given $a_1,a_2 \in \R$ such that $a_1+a_2=1$ and $f_1,f_2\in Y$ we want to show that
	\begin{equation}\label{eq:affine}
		\big( g\cdot (a_1f_1+a_2f_2) \big) (\gamma )=  a_1(g\cdot f_1) (\gamma ) + a_2 (g\cdot f_2) (\gamma ), \forall \gamma \in B. 
	\end{equation}
	
	Corresponding to the element $g\in G$ there is a splitting $B= \bigsqcup_{h\in \Lambda} B_{g,h}$. It suffices to prove \eqref{eq:affine} on each $B_{g,h}$ with $h\in \Lambda$. When $\gamma \in B_{g,h}$, the left hand side of \eqref{eq:affine} becomes:
\[
h\cdot (a_1 f_1+ a_2f_2)(h^{-1}\gamma g^{-1}) = h\cdot \left[ a_1 f_1(h^{-1}\gamma g^{-1}) + a_2 f_2(h^{-1}\gamma g^{-1}) \right] 
\]

In the equality above, we applied the definition of the vector space structure of $Y$, which is the usual one, for the restrictions to the fundamental domain $B$. We now use the fact that $h\in \Lambda$ is an affine transformation of $X$, and thus continue with 

\[
=  a_1 h\cdot f_1(h^{-1}\gamma g^{-1}) + a_2 h\cdot f_2(h^{-1}\gamma g^{-1}) = a_1 (g\cdot f_1) (\gamma ) + a_2 (g\cdot f_2) (\gamma ).     
\] 
\comment 
\green{This is much nicer than before, here's an alternative without $B_{g,h}$'s, but I don't mind which:
Given $g \in G$ and $a,a' \in \R$ such that $a+a'=1$ and $f,f'\in Y$ we want to show that
	\[
	g\cdot (a f+ a' f') (\gamma )= a (g\cdot f) (\gamma ) + a' (g\cdot f') (\gamma ), \forall \gamma \in B.
	\]
We have
\begin{align*}
	g\cdot (a f+ a' f') (\gamma )
	& = \alpha(g,\gamma)^{-1} \cdot (a f+ a' f')(g \bullet \gamma)
	\\ & = \alpha(g,\gamma)^{-1} \cdot \left( a f(g\bullet \gamma) + a' f'(g \bullet \gamma) \right)
\intertext{using the vector space structure of $Y$ which is the usual one on the fundamental domain $B$, so as the action of $\Lambda$ on $X$ is affine:}
	& = a \alpha(g,\gamma)^{-1} \cdot f(g\bullet \gamma) + a' \alpha(g,\gamma)^{-1} \cdot f'(g \bullet \gamma)
	\\ & = a (g \cdot f)(\gamma) + a' (g \cdot f')(\gamma).
\end{align*}
}
\endcomment

\medskip  

\noindent \emph{The action on $Y$ is uniformly Lipschitz.} \quad For $g \in G$ and $f_1,f_2 \in Y$, using that the action of $\Lambda$ on $X$ is $L$-Lipschitz, and the substitution $\gamma'=\gamma g^{-1}$,
	\begin{align*}
		\|g\cdot f_1 - g\cdot f_2 \|_Y^p
		& = \int_B \big\|(g\cdot f_1)(\gamma) - (g\cdot f_2 )(\gamma )\big\|_X^p \,d\mu(\gamma)
		\\ & = \sum_{h\in \Lambda} \int_{B_{g,h}} \|(g\cdot f_1)(\gamma) - (g\cdot f_2) (\gamma )\|_X^p \,d\mu(\gamma)
		\\ & = \sum_{h\in \Lambda} \int_{B_{g,h}} \|h \cdot f_1(h^{-1}\gamma g^{-1}) - h \cdot f_2(h^{-1}\gamma g^{-1}) \|_X^p \,d\mu(\gamma)
		\\ &  \leq L^p  \sum_{h\in \Lambda} \int_{B_{g,h}} \|f_1(h^{-1}\gamma g^{-1}) - f_2(h^{-1}\gamma g^{-1}) \|_X^p \,d\mu(\gamma)
		\\ & \leq L^p \sum_{h\in \Lambda}  \int_{h^{-1} B_{g,h} g^{-1}} \|f_1 (\gamma')-f_2(\gamma') \|_X^p \,d\mu(\gamma')
		\\ &  = L^p \int_{B} \|f_1(\gamma')-f_2 (\gamma') \|_X^p \,d\mu(\gamma').
	\end{align*} 

The last equality above is due to the fact that the subsets $\{h^{-1}B_{g,h} g^{-1}\}_{h \in \Lambda}$ partition $B$ as they are the image of the partition $\{B_{g,h}\}_{h\in \Lambda}$ under the action of $g$ on $B$.

\medskip

\noindent \emph{The action $G \times Y \to Y$ is continuous.}\quad For each fixed $g$, the map $f \mapsto g \cdot f$ from $Y$ to $Y$ is uniformly Lipschitz, so it suffices to show that for each $f \in Y$, the orbit map $G \to Y, g \mapsto g\cdot f$ is continuous.

First, we show continuity at $g=\id_G$. Let $\epsilon > 0$ and $f \in Y$ be arbitrary.
	Recall that we can identify $Y$ with $L^p(B,\mu;X)$.
	Since the interior of $B$ is an open subset of $G$, and $\mu(\partial B)=0$, we can find a continuous function $f_1$ with compact support contained in the interior of $B$, and with $d_Y(f,f_1)<\epsilon/(3L)$. By compactness, there exists $\delta>0$ so that for all $a \in G$ with $d_G(\id_G,a)<\delta$ and for all $\gamma$ in the support of $f_1$ we have $d_X(f_1(\gamma),f_1(\gamma a^{-1}))<\epsilon/(3\mu(B)^{1/p})$ and that $\gamma a^{-1} \in B$.
	Thus
	\begin{align*}
		d_Y(f, a\cdot f)
		& \leq d_Y(f,f_1)+d_Y(f_1,a\cdot f_1)+d_Y(a \cdot f_1,a\cdot f)
		\\ & \leq \frac{\epsilon}{3L}+\left(\int_B\|f_1(\gamma)-f_1(\gamma a^{-1})\|_X^p \,d\mu(\gamma)\right)^{1/p} + Ld_Y(f_1,f)
		\\ & \leq \frac{\epsilon}{3L} + \frac{\epsilon}{3}+\frac{\epsilon}{3} < \epsilon.
	\end{align*}
	Thus $G \to Y, g \mapsto g \cdot f$ is continuous at $g=\id_G$.

	The general case follows: given $g \in G$, $f \in Y$, $\epsilon >0$, choose $\delta>0$ so that $d_G(\id_G,a) < \delta$ implies that $d_Y(f, a\cdot f) < \epsilon/L$.
	Then $d_G(g,g')< \delta$ implies that 
	$d_G(\id,g^{-1}g')<\delta$,
	so $d_Y(f, (g^{-1}g')\cdot f) < \epsilon/L$,
	so $d_Y(g\cdot f, g'\cdot f) < \epsilon$.

\medskip

We now proceed to the comparison between the orbits of $\Lambda$ in $X$ and the orbits of $G$ in $Y$. 

\noindent\emph{\eqref{item:unbded} Unbounded orbits for $G$.}\quad 
For every $g \in G$, by \eqref{eq:lattice-gf}
	\begin{align*}
		\|g \cdot 0_Y\|_Y^p
		& = \int_B \| \alpha(g,\gamma)^{-1} \cdot 0_Y(g \bullet \gamma)\|_X^p \,d\mu(\gamma).
	\end{align*}  
Suppose the $G$-orbit of $0_Y$ is bounded.
We wish to show the $\Lambda$-orbit of $0_X$ is bounded to obtain a contradiction.
If $\Lambda$ is finite index in $G$ this is straightforward, but for the general case we require the following result of Ozawa, translated into our notation.
\begin{theorem}[\cite{Ozawa-11-qhom-rigid}*{Corollary 9}]
	\label{thm:ozawa}
	Let $G$ be a second countable locally compact group $G$ with lattice $\Lambda$, fundamental domain $B$ and cocycle $\alpha: G \times B \to \Lambda$.
	Suppose $\ell:\Lambda \to \R_{\geq 0}$ is a function such that $\ell(hh')\leq C(\ell(h)+\ell(h'))$ for some fixed $C \geq 1$ and all $h,h' \in \Lambda$.
	Let $L : G \to [0,\infty]$ be defined by
	\[
		L(g) := \int_B \ell(\alpha(g,\gamma)) d\mu(\gamma).
	\]
	If $L$ is essentially bounded, then $\ell$ is bounded.
\end{theorem}
Note that Ozawa states his corollary for $C=1$ but the proof shows the $C>1$ case also since \cite{Ozawa-11-qhom-rigid}*{Theorem 8} allows this. Ozawa also works with $G/\Lambda$ rather than $\Lambda\backslash G$, and a section $\sigma:G/\Lambda \to G$ rather than a fundamental domain $\sigma(G/\Lambda)$, but these changes are superficial.

Proceeding with our proof, let $\ell:\Lambda \to \R_{\geq 0}$ be defined by $\ell(h) = \|h \cdot 0_X\|_X$.  Note that for any $h,h' \in \Lambda$,
\begin{align*}
	\ell(hh') 
	& = \|hh' \cdot 0_X\|_X
	= \|h\cdot(h' \cdot 0_X) - h \cdot 0_X + h \cdot 0_X\|_X 
	\\ & \leq \|h \cdot(h' \cdot 0_X - 0_X)\|_X + \|h \cdot 0_X\|_X
	\leq L \ell(h')+\ell(h).
\end{align*}
Moreover
\begin{align*}
	L(g)
	& = \int_B \ell(\alpha(g,\gamma)) d\mu(\gamma)
	= \int_B \|\alpha(g,\gamma) \cdot 0_X\|_X d\mu(\gamma)
	\\ & \leq L \int_B \|\alpha(g,\gamma)^{-1} \cdot 0_X\|_X d\mu(\gamma)
	\leq L \left( \int_B \|\alpha(g,\gamma)^{-1} \cdot 0_X \|_X^p \right)^{1/p}
	\\ & = L \|g \cdot 0_Y\|_Y.
\end{align*}
As the $G$ orbit of $0_Y$ is bounded, $L$ is essentially bounded, hence Theorem~\ref{thm:ozawa} implies that $\ell$ is bounded.  Therefore the $\Lambda$ orbit of $0_X$ is bounded, contradiction.

\medskip

\noindent\emph{\eqref{item:proper} Proper orbits.}\quad 
Let $K \subset B$ be a measurable compact subset so that $\mu(K) \geq \frac{2}{3} \mu(B)$.  When $\Lambda$ is a cocompact lattice, we take $K=B$.  
To every element $g\in G$ we can associate the set $K_g = \{ \gamma \in B: \gamma\in K, g \bullet \gamma \in K\}$, which satisfies $\mu(K_g) \geq \frac{1}{3}\mu(B)$.

Since $K_g \subset K$ is compact, there exists $C>0$ so that for any $g \in G$, any $a \in K_gg^{-1}$ satisfies $\|g\|_G - C \leq \|a\|_G \leq \|g\|_G + C$.

Thus, if $\gamma \in K_g$, as $\alpha(g,\gamma)$ translates $\gamma g^{-1}$ back to $g \bullet \gamma = \alpha(g,\gamma)\gamma g^{-1} \in K$, we must have 
	$|\alpha(g,\gamma) |_\Lambda \geq \frac{1}{C'} \|\alpha(g,\gamma)\|_G -C'   \geq \frac{1}{C'}(\|g\|_G-C)-C'$ for some $C'$.

	Therefore,
	\begin{align*}
		\|g \cdot 0_Y\|_Y^p
		& = \int_B \| \alpha(g,\gamma)^{-1} \cdot 0_Y(g \bullet \gamma)\|_X^p \,d\mu(\gamma)
		\\ & \geq \int_{K_g} \| \alpha(g,\gamma)^{-1} \cdot 0\|_X^p \,d\mu(\gamma)
		\\ & \geq \mu(K_g) \frac{1}{C''}\rho\left( \tfrac{1}{C'}(\|g\|_G-C)-C'\right)-C''
	\end{align*}
	for some $C''$.
\end{proof}

\medskip

\section{Construction}\label{sec:construction} 
In this section, we describe the basic tools needed to build our actions.
\subsection{Linear representations and derivations}

As in the introduction, let $E$ be a normed vector space, let $L(E,E)$ be the algebra of linear maps from $E$ to $E$, and $\cB (E )$ the sub-algebra of bounded operators.

We endow $E \oplus E$ with the norm $\|(x,y)\|=\|x\|+\|y\|$. The particular product norm will not be relevant most of the time.

Let $G$ be a group. We begin by describing how to build a linear representation on $E \oplus E$ using a linear representation on $E$ and a derivation with respect to it.  (Compare for example \cite{pisier-01-similarity-book}*{Proof of Theorem 2.1}.)

Consider two maps $\pi :G \to L(E,E)$, $D: G \to L(E,E)$, and an associated map $\pi_D : G \to L (E \oplus E, E \oplus E)$ defined by
\[
	\pi_D (g) = 
	\mtx{
		\pi (g) & D(g) \\ 
		0 & \pi (g)
	} .
\]

\begin{definition}\label{def:deriv}
Given $\pi :G \to L(E,E)$ a linear representation, a map $D:G \to L(E,E)$ is an \emph{algebraic derivation with respect to} $\pi$ if it satisfies the Leibniz rule
$$
	D(gh) = D(g) \pi (h) + \pi(g) D(h), \text{ for all } g,h \in G.
$$

If moreover $\pi$ is uniformly bounded and $D$ maps into $\cB(E)$, we call $D$ a \emph{derivation with respect to} $\pi$.
If moreover $\sup_{g\in G} \|D(g)\|_{op} < \infty$, we call the derivation $D$ \emph{bounded}.
\end{definition}

Note that when $\pi$ is the trivial representation, $D$ becomes a group homomorphism, obviously factoring through the abelianization of $G$. 

\comment
\green{Skip para?:}In particular, when $G$ has property (T) and a non-trivial representation $\pi$ is isolated from the trivial, this should give some separation of every $D$ derivation with respect to $\pi$ from group homomorphisms. 
\endcomment

\begin{proposition}
	\label{prop:deriv}
We have that:
\begin{enumerate}
\item $\pi_D$ is a linear representation $\Leftrightarrow $ $\pi$ is a linear representation and $D$ is an algebraic derivation with respect to $\pi$.

\item 
    Assuming (1) holds, for any $g \in G$, 
	\[
		\|\pi_D(g)\|_{op} 
		\leq  \|\pi(g)\|_{op}+\|D(g)\|_{op}
		\leq 2 \|\pi_D(g)\|_{op}.
	\]
	Hence, 
	$\pi_D$ is a uniformly bounded representation $\Leftrightarrow $ $\pi$ is uniformly bounded and $D$ is a bounded derivation with respect to $\pi$.
\end{enumerate}
\end{proposition}
\begin{proof}
	(1) For any $g,h \in G$, we have
	\[ \pi_D(g)\pi_D(h) 
	= \mtx{\pi(g)\pi(h) & \pi(g)D(h)+D(g)\pi(h) \\ 0 & \pi(g)\pi(h)}.
	\]

	(2) We have the following bounds for $g \in G, x,y \in E$:
	\begin{align*}
		\left\| \pi_D(g) \mtx{x \\ y} \right\|
		& = \|\pi(g)x +D(g)y \| + \|\pi(g) y\|
		\\ & \leq \|\pi(g)\|_{op} \left(\|x\|+\|y\|\right) + \|D(g)\|_{op} \|y\|
		\\ & \leq \left( \|\pi(g)\|_{op} + \|D(g)\|_{op} \right) \left\|\mtx{x \\ y}\right\|,
	\end{align*}
	whence $\|\pi_D(g)\|_{op} \leq  \|\pi(g)\|_{op}+\|D(g)\|_{op}$. For the second inequality, we note that
	\begin{align*}
		\|D(g)y\|+\|\pi(g)y\| = \left\| \pi_D(g) \mtx{0 \\ y} \right\|
		\leq \|\pi_D(g)\|_{op} \|y\|,
	\end{align*}
	from which we can deduce that $\|D(g)\|_{op} \leq \|\pi_D(g)\|_{op}$ and
	 $\|\pi(g)\|_{op} \leq \|\pi_D(g)\|_{op}$.
\end{proof}

Now consider a function $\zeta_D :G \to E \oplus E $,
\[
\zeta_D (g) = \mtx{ \alpha_D (g) \\ \beta_D (g) }.
\]
\begin{lemma}\label{lem:deriv-to-cocycle-for-EplusE}
	The function $\zeta_D$ defines a cocycle for the representation $\pi_D$ if and only if
	$\beta_D$ is a cocycle for $\pi$, and for all $g,h\in G$,
	\begin{equation}\label{alphaD}
\alpha_D (gh)=  \alpha_D (g) +  \pi (g)\alpha_D (h)+  D(g) \beta_D (h).	
	\end{equation}
\end{lemma}
\begin{proof}
	The cocycle equation $\zeta_D (gh) = \zeta_D (g) + \pi_D(g) \zeta_D (h)$
is equivalent to
	\[
\left\{ 
\begin{array}{cccccc}
	\alpha_D (gh)= & \alpha_D (g) & + & \pi (g)\alpha_D (h) & + & D(g) \beta_D (h); \\

	\beta_D (gh)= & \beta_D (g)& + & \pi(g) \beta_D (h). & \\
\end{array}\right.\qedhere
	\]
\end{proof}
To find a cocycle for $\pi_D$, we want $\beta_D : G \to E $ to be a cocycle for $\pi$ independent of $D$, so we can simply write $\beta$ instead of $\beta_D$.
However, the obstruction to $\alpha_D$ being a cocycle for $\pi$ depends on $D$ and $\beta$.

\subsection{Construction}

Our idea is to reverse this algebra: for a suitable representation $\pi$, cocycle $\beta$ for $\pi$, and function $\alpha$, we find a derivation 
$D$ so that $(\alpha,\beta)$ is a cocycle for $\pi_D$, equivalently so that $\beta$ is a cocycle for $\pi$ and \eqref{alphaD} is satisfied.
In the following, the span of a collection of vectors has the usual meaning of subspace of all finite linear combinations of the given vectors.
\begin{lemma}\label{lem:alg-deriv}
	Let $H$ be a subgroup of a group $G$.
	Suppose $\beta : G \to E$ is a cocycle for a 
	representation $\pi:G \to L(E, E)$, with 
	$\beta(gh)=\beta(g)$ for all $g\in G, h\in H$; therefore $\beta([g])$, for $[g]\in G/H$, is well-defined.
	Suppose further that
	$\{ \beta([g]) : [g]\in G/H , [g]\neq [\id]\}$ is linearly independent.

	Let $E' = \Span \{ \beta(g) : g\in G \}$ and let $\alpha:G \to E$ be an arbitrary map with $\alpha(\id)=0$ and $\alpha(gh)=\alpha(g)$ for all $g\in G, h\in H$. 
	Define $D(g) : E' \to E$ by 
	\begin{equation}\label{def:D}
		D(g)\beta([g']) = \alpha(gg')-\alpha(g)-\pi(g)\alpha(g'), \forall g' \in G.
	\end{equation}
	Then $D$ satisfies
	\[
		 D(gh) = D(g) \pi (h) + \pi(g) D(h)
	\]
	for all $g,h \in G$, considered as functions in $L(E',E)$.
\end{lemma}

The conditions in this lemma are natural.   The domain of $D(g)$ is naturally $E'$, which is a $\pi(G)$-invariant subspace, as $\pi(g)\beta(g') = \beta(gg')-\beta(g)$ for all $g,g'\in G$.  To have $D(g)$ map $E'\to E'$, it then suffices that $\alpha$ maps into $E'$.  
To make $D(g)$ well-defined, we define it on a linearly independent set. Note that $\beta(\id)=0$ as $\beta$ is a cocycle. We therefore want $\alpha(\id)=0$ as well: indeed, in \eqref{def:D}, when we let $g'=\id$ we obtain $D(g)0=-\pi(g)\alpha(\id)$, and since $D(g)$ and $\pi (g)$ are linear, the conclusion follows.

In the particular case when $\alpha$ maps into $E' \subset E$, this lemma states that $D$ is an algebraic derivation for $\pi$ restricted to $\pi:G \to L(E', E')$.  We will later apply the lemma when $E$ is the closure of $E'$, and use boundedness to extend $D$ to be a derivation on the whole space $E$.

\begin{proof}[Proof of Lemma~\ref{lem:alg-deriv}]
It suffices if the equality between linear maps 
	\[
D(ab) = D(a) \pi (b) + \pi (a) D(b)
	\] is satisfied on the family of vectors $\{\beta(g) :  g\in G \}$.

For any $g\in G$, the left hand side becomes
\begin{equation*}
D(ab) \beta(g) = \alpha (abg) - \alpha (ab) - \pi (ab)\alpha (g),
\end{equation*}
while the right hand side becomes
\begin{align*}
	& [D(a) \pi (b) + \pi (a) D(b)] \beta(g) 
	\\ & \quad = D(a) (\beta(bg)- \beta(b)) + \pi (a)\big(\alpha (bg) - \alpha (b) - \pi (b)\alpha (g)\big) 
	\\ & \quad = \alpha (abg) - \alpha (a) - \pi (a)\alpha (bg)  - [\alpha (ab) - \alpha (a) - \pi (a)\alpha (b)] 
	\\ & \qquad + \pi (a)\alpha (bg) - \pi(a)\alpha (b) - \pi(a)\pi (b)\alpha (g)
\\ 
 & \quad =  \alpha (abg) - \alpha (ab)  - \pi (ab)\alpha (g). \qedhere
\end{align*}
\end{proof}

In the particular construction we now describe, we will always consider $\beta$ to be a coboundary; this is no real loss of generality, because most of the classes of groups to which we will apply our results contain large subclasses of groups with Property (T). 
More significantly, we will restrict our attention to Banach spaces that are isomorphic to $\ell^1$ (sub)spaces, where we are able to show that the algebraic derivation $D$ above can be extended to a bounded derivation.

\medskip

Given a discrete, countable (infinite) set $Q$, denote by $\ell^1 Q$ the standard $\ell^1$-Banach space of functions $\{f:Q \to \R : \sum_{q\in Q} |f(q)| <\infty\}$. 
Denote the subspace of zero sum functions by:
\[
	\ell^1_0 Q = \Bigg\{ f \in \ell^1 Q :  \sum_{q \in Q}f(q) = 0 \Bigg\},
\]
 and its intersection with the subspace of finitely supported functions by:
\[
	\ell^1_{00} Q = \Bigg\{ f \in \ell^1 Q :  \sum_{q\in Q}f(q) = 0, |\{q:f(q)\neq 0\}|<\infty \Bigg\}.
\]

For the sake of completeness, we include the following easy observation, with proof.
\begin{lemma}\label{lem:l100-l10}
	The subspace $\ell^1_0 Q$ is the closure of the subspace $\ell^1_{00} Q$.
\end{lemma}
\begin{proof}
	If $(f_i)$ is a sequence in $\ell^1_{00} Q$ converging to $f \in \ell^1(Q)$, then 
	\[
		\Bigg| \sum_{q\in Q} f(q) \Bigg|
		= \Bigg| \sum_{q \in Q} (f(q)-f_i(q)) \Bigg|
		\leq \sum_{q \in Q} |f(q)-f_i(q)|
		= \|f-f_i\| \to 0
	\]
	as $i\to \infty$, giving that $\overline{\ell^1_{00}Q} \subset \ell^1_0 Q$.  Conversely $\ell^1_0 Q \subset \overline{\ell^1_{00}Q}$ since $\ell^1_{00}Q$ is dense in $\ell^1_0 Q$: given $f \in \ell^1_0 Q$ and $\epsilon>0$, choose $Q' \subset Q$ finite with $\sum_{q \in Q\setminus Q'} |f(q)|<\epsilon/2$.  Define $g \in \ell^1_{00} Q$ by setting $g=0$ on $Q\setminus Q'$, and setting $g=f$ in $Q'$, except for one point where the value of $g$ is adjusted by $\sum_{q \in Q\setminus Q'} f(q) \in (-\epsilon/2,\epsilon/2)$.  This function satisfies $\|f-g\| < \epsilon$.  Therefore the claim holds.
\end{proof}

The spaces $\ell^1_0 Q$ and $\ell^1$ are isomorphic, in fact:
\begin{lemma}
	\label{lem:l10-l1}
	The Banach--Mazur distance between $\ell^1_0 = \ell^1_0 \N$ and $\ell^1 = \ell^1\N$ 
	\[
	d_{BM}(\ell^1_0, \ell^1) = \inf \{ \|F\|_{op} \|F^{-1}\|_{op} : F: \ell^1_0 \to \ell^1 \text{ isomorphism} \}
	\]
	satisfies $d_{BM}(\ell^1_0,\ell^1) \in (1,2]$, with the upper bound attained by an isomorphism $F$.
\end{lemma}
As mentioned before, the proof we provide below that $d_{BM}(\ell^1_0,\ell^1)>1$ was explained to us by W.\ Johnson.
\begin{remark}\label{rmk:lip-const-banach-mazur-dist}
	Thus, if we build an affine uniformly $L$-Lipschitz action of a group on (the $\ell^1$-sum of finitely many copies of) $\ell^1_0 Q$, conjugating by an appropriate isomorphism we can find an affine uniformly Lipschitz action on $\ell^1$ with Lipschitz constant $(d_{BM}(\ell^1_0,\ell^1)+\epsilon)L$.
\end{remark}
\begin{proof}[Proof of Lemma~\ref{lem:l10-l1}]
	Let $\N = \{0,1,2\ldots\}$.
	Define $F:\ell^1_0 \to \ell^1(\N\setminus \{0\})$ by $F(f)=f|_{\N\setminus\{0\}}$.  This clearly satisfies $\|F(f)\| \leq \|f\|$, and has an inverse $F^{-1}:\ell^1(\N\setminus\{0\})\to \ell^1_0$ defined by 
	\begin{equation*}
		F^{-1}(f)(q)= \begin{cases} f(q) & \text{if } q\neq 0, \\
						-\sum_{q\neq 0}f(q) & \text{if } q=0,
				\end{cases}
	\end{equation*}
	which satisfies $\|F^{-1}(f)\| \leq 2\|f\|$, so $F$ is the isomorphism required to show the upper bound $d_{BM}(\ell^1_0,\ell^1) \leq 2$.
	
	For the lower bound, consider the subspace
	\[ 
	X = \{(x_0,x_1,\ldots)\in \ell^1_0 : x_i=0,\, \forall i \geq 3\}. \]
	This is a contractively complemented subspace of $\ell^1_0$, i.e.\ there is a projection 
	\[
		P:\ell^1_0\to\ell^1_0, P((x_i)) = \Big(x_0,x_1, \sum_{i \geq 2} x_i, 0,\dots  \Big)
	\]
	with image $X$ and $\|P\|_{op}=1$.

	Suppose $d_{BM}(\ell^1_0,\ell^1)=1$.
	Then there is a sequence of isomorphisms $F_i:\ell^1_0 \to \ell^1$ with $\|F_i\|_{op}=1$ and $\|F_i^{-1}\|_{op}\to 1$.
	Let $X_i = F_i(X)$, and consider the sequence of projections
	$F_i P F_i^{-1} : \ell^1 \to \ell^1$ with image $X_i$.
	Taking an ultralimit of these spaces and maps we get a contractive projection
	$P_\omega : E \to E$ on a space $E$ isometrically isomorphic to some $L^1$ space, with image $X_\omega$ isometrically isomorphic to $X$.

	Any contractively complemented subspace of any $L^1$ space is (isometrically isomorphic to) an $L^1$ space, see e.g.~\cite{Bernau-Lacey-74-contractive-lp}.
	But $X$ is not isometric to $\R^2$ with the $L^1$-metric -- for example the unit ball is a hexagon rather than a square -- contradiction.  So $d_{BM}(\ell^1_0,\ell^1)>1$.
\end{proof}
The following statement is our main tool for building affine uniformly Lipschitz actions on $\ell^1$.
For $q \in Q$, let $\delta_q \in \ell^1 Q$ denote the indicator function of $q$.
\begin{proposition}\label{prop:unifLip-action}
	Consider a group $G$ acting transitively on a discrete set $Q$, a point $x_0 \in Q$ and $G_{x_0}$ its stabilizer.
	Let $\pi$ be the orthogonal representation of $G$ on $\ell^1 Q$ defined by $\pi(g)(f)(\cdot) = f(g^{-1}\cdot)$, which restricts to an orthogonal representation on $\ell^1_0 Q$.
	Let $\beta:G \to \ell^1_0 Q$ be the coboundary $\beta(g) = \delta_{x_0}-\pi(g)\delta_{x_0} = \delta_{x_0}-\delta_{gx_0}$.
	Suppose $\alpha:G \to \ell^1_0 Q$ is a function with 
	$\alpha (\id) =0$ and $\alpha(gh)=\alpha(g)$ for all $g\in G, h\in G_{x_0}$, that moreover is a \emph{quasi-cocycle} for $\pi$, that is:
	\[ \Delta(\alpha) := \sup_{g,g'\in G} \| \alpha(gg')-\alpha(g)-\pi(g)\alpha(g') \| < \infty.
		\]

	Then there is an affine uniformly $(1+\Delta(\alpha))$-Lipschitz action of $G$ on the space $\ell^1_0 Q \oplus \ell^1_0 Q$ for which $(\alpha,\beta)$ is the cocycle.

	Consequently, by rescaling $\alpha$ and conjugating the action by an isomorphism, for any $L > d_{BM}(\ell^1_0, \ell^1)$, e.g.\ any $L>2$, we can find an affine uniformly $L$-Lipschitz action on $\ell^1$ with a cocycle having norm comparable to $(\alpha,\beta)$.
\end{proposition}

\begin{proof}
	First, $\beta(gh)=\beta(g)$ for all $g\in G, h \in G_{x_0}$, and $\{ \beta([g]) : [g]\in G/G_{x_0}, [g]\neq[\id] \}$ are linearly independent vectors whose span is $\ell^1_{00} Q$.

	Lemma~\ref{lem:alg-deriv} applies to define, for each $g\in G$, linear maps $D(g): \ell^1_{00} Q \to \ell^1_0 Q$ that satisfy
	\begin{equation}\label{eq:leibniz}
		D(gh)=D(g)\pi(h)+\pi(g)D(h), \text{for all } g,h \in G.
	\end{equation}

	For any $f \in \ell^1_{00} Q$, we can write
	\begin{equation*}
		f = \sum_{[g] \in G/G_{x_0}, [g]\neq [\id]} -f(gx_0)\beta(g).
	\end{equation*}
	
	We now bound $D$:  
	\begin{align*}
		D(g)f
		& = D(g)\left( \sum_{[g'] \in G/G_{x_0}, [g']\neq [\id]} -f(g'x_0)\beta(g')\right)
		\\ & = \sum_{[g'] \in G/G_{x_0}, [g']\neq [\id]} -f(g'x_0)D(g)\beta(g'),
	\end{align*}
	thus
	\begin{align*}
		\| D(g)f \| 
		& \leq \sum_{[g'] \in G/G_{x_0}, [g']\neq [\id]} |f(g'x_0)| \, \Delta(\alpha) 
		\\ & = \Delta(\alpha) \sum_{x \in Q\setminus\{x_0\}} |f(x)|
		 \leq \Delta(\alpha) \|f\|,
	\end{align*}
	i.e.\ the linear maps $D(g): \ell^1_{00} Q \to \ell^1_0 Q$ are all bounded with norm at most $\Delta(\alpha)$. Therefore, for every $g\in G$, we can extend $D(g)$ to a linear operator on $\ell^1_0 Q = \overline{\ell^1_{00}Q}$, with $\|D(g)\|_{op} \leq \Delta(\alpha)$.

	Since \eqref{eq:leibniz} is true on a dense subspace of $\ell^1_0 Q$, and both sides are bounded operators, it holds considering the two sides as operators on $\ell^1_0 Q$, so $D$ is a bounded derivation on $\ell^1_0 Q$ with $\|D(g)\|_{op} \leq \Delta(\alpha)$ for all $g \in G$.

	Proposition~\ref{prop:deriv} and Lemma~\ref{lem:deriv-to-cocycle-for-EplusE} then apply to give a representation $\pi_D$ on $\ell^1_0 Q \oplus \ell^1_0 Q$ which admits $(\alpha,\beta)$ as a cocycle, with $\pi_D$ uniformly bounded by $1+\Delta(\alpha)$.  
	
	Finally, given $\epsilon>0$, we can apply the argument above to $\epsilon\alpha$ which satisfies $\Delta(\epsilon\alpha)=\epsilon\Delta(\alpha)$ to find an action on $\ell^1_0Q \oplus \ell^1_0Q$.
	As in Lemma~\ref{lem:l10-l1}, we can find an isomorphism $F:\ell^1_0 Q \to \ell^1$ such that $\|F\|_{op}\|F^{-1}\|_{op} \leq d_{BM}(\ell^1_0,\ell^1) + \epsilon$, and conjugate the action above by $F \oplus F$ to find an action
	\[
	\mtx{x \\ y } \mapsto F \pi_D(g)\left( F^{-1}\mtx{x \\ y} \right) + \mtx{F \epsilon\alpha(g) \\ F \beta(g) }
	\]
	on $\ell^1 \oplus \ell^1=\ell^1$ which is uniformly $L$-Lipschitz for any $L$ satisfying
	\[
		\|F\|_{op} (1 + \epsilon \Delta(\alpha)) \|F^{-1}\|_{op}
		\leq (d_{BM}(\ell^1_0,\ell^1)+\epsilon)(1+\epsilon \Delta(\alpha)) \leq L,
	\]
	and has cocycle $(\epsilon F\alpha(g),F\beta(g))$ comparable to $(\alpha,\beta)$.
\end{proof}

\section{Quasi-cocycles for acylindrically hyperbolic groups}\label{sec:acylind}

In this section, we use a quasi-cocycle construction of Bestvina--Bromberg--Fujiwara \cite{BBF-bdd-cohom-v2} and Hull--Osin~\cite{hull-osin} for acylindrically hyperbolic groups to show the following.
\begin{varthm}[Theorem \ref{thm:acyl-hyp-unbounded-l1-action}]
Any acylindrically hyperbolic group $G$ admits an affine uniformly $(2+\epsilon)$-Lipschitz action on $\ell^1$ with unbounded orbits for any $\epsilon>0$, hence likewise for $L^1=L^1([0,1])$.
\end{varthm}
We begin by recalling Bestvina--Bromberg--Fujiwara's extension of Brooks' counting quasi-morphism.

\begin{notation}\label{notat:eps}
Let $T$ be a real tree.

For $x,y \in T$ we let $[x,y]$ denote the oriented geodesic segment from $x$ to $y$.
For $x,y,p,q \in T$ we write $[x,y] \subseg [p,q]$ if $[x,y]$ is contained in $[p,q]$ and if their orientations agree.  
\end{notation}

Consider the free group $F_2= \langle a,b \rangle$, $w \in F_2$ a reduced non-empty word, $\tilde{E}$ a normed vector space with an isometric linear $F_2$-action, and $e \in \tilde{E}$ a (non-zero) vector.

In the Cayley graph of $F_2$ with respect to $\{ a,b,a^{-1},b^{-1} \}$, using the notation introduced in \ref{notat:eps} we define, for every $g\in F_2$, 
\begin{align*}
	w_+(g) & = \{ h \in F_2 : [h,hw] \subseg [\id,g] \},
\\	w_-(g) & = \{ h \in F_2 : [h,hw] \subseg [g,\id] \},
\end{align*}
and
\begin{equation}\label{eq:eta}
	\eta=\eta_{w,e}:F_2\to \tilde{E}, \eta(g) = \sum_{h\in w_+(g)}h\cdot e - \sum_{h\in w_-(g)}h\cdot e.
\end{equation}
A short argument gives the following.
\begin{proposition}[\cite{BBF-bdd-cohom-v2}*{Proposition 2.1}]\label{prop:BBF-free-qcocycle}
	The function $\eta$ is a quasi-cocycle that is moreover anti-symmetric, i.e.\ $\eta(g^{-1}) = -g^{-1} \cdot \eta(g)$.
\end{proposition}
The idea is that, when computing $\eta(gg')-\eta(g)-g\cdot \eta(g')$, all terms cancel apart from a bounded number involving $h$ located near the centre of the tripod with vertices $\id,g,gg'$.

We may now prove the theorem.
\begin{proof}[Proof of Theorem~\ref{thm:acyl-hyp-unbounded-l1-action}] We aim to apply Proposition~\ref{prop:unifLip-action} to the action of $G$ on itself.

Following \cite{BBF-bdd-cohom-v2}*{Proof of Corollary 1.2}, let $K$ be the maximal finite normal subgroup of $G$.
	By Osin~\cite{osin-acylindrically-hyp}*{Theorem 1.2} and Dahmani--Guirardel--Osin~\cite{DGO-hyp-embedded-rotating-book}*{Theorem 2.24} there exists a hyperbolically embedded $F_2\times K \leq G$ (compare \cite{BBF-bdd-cohom-v2}*{Theorem 4.5}).  
	
	Let $F_2 = \langle a,b \rangle$. Consider $\tilde{E} = \ell^1_0(F_2)$ with the usual left $F_2$-action, the vector $e = \delta_{\id}-\delta_{a^2}\in \tilde{E}$ and the word $w=ab$, and define the quasi-cocycle $\eta=\eta_{w,e}$ using the formula \eqref{eq:eta}.
	As in \cite{BBF-bdd-cohom-v2}*{Example 2.3}, $\eta$ is unbounded:
	\[ \eta((ab)^n) = (\id+ab+\cdots+(ab)^{n-1})e
		= \sum_{i=0}^{n-1} \left( \delta_{(ab)^i}-\delta_{(ab)^ia^2} \right)\]
	so $\|\eta((ab)^n)\|= 2n \to \infty$ as $n\to\infty$.
	Extend $\eta$ trivially on $K$, that is, for $(g,k),(h,l) \in F_2\times K$, let $\eta((g,k))((h,l)) = \eta(g)(h)$.
	This defines an unbounded quasi-cocycle on $\ell^1_0(F_2 \times K)$.

	The subgroup $F_2\times K$ is hyperbolically embedded in $G$, and $\ell^1_0(F_2 \times K)$ is a $(F_2\times K)$-submodule of $\ell^1_0(G)$, therefore, according to \cite{hull-osin}*{Theorem 1.4}, $\eta$ extends to an (anti-symmetric unbounded) quasi-cocycle $\alpha:G \to \ell^1_0(G)$.

	Proposition~\ref{prop:unifLip-action} applied to $Q=G$ with the action on itself by left multiplication, and $x_0=\id$ with trivial stabilizer yields an affine  uniformly Lipschitz action $G\acts \ell^1_0(G)\oplus \ell^1_0(G)$ with cocycle $\zeta =(\alpha, \beta)$ for $\alpha$ as above and $\beta (g)= \delta_\id -\delta_g$.
	Since $\alpha$ is unbounded, $G\acts \ell^1_0(G)\oplus \ell^1_0(G)$ has unbounded orbits.

	As remarked in Proposition~\ref{prop:unifLip-action}, up to rescaling $\alpha$ and conjugating by an isomorphism, for any $\epsilon>0$ we can find an affine uniformly $(2+\epsilon)$-Lipschitz action of $G$ on $\ell^1$ with unbounded orbits.
	The $L^1$ statement follows from Proposition~\ref{prop:l1-to-L1}.
\end{proof}

\section{Quasi-trees and quasi-cocycle estimates}
\label{sec:quasi-trees-l1-bound}

Bestvina--Bromberg--Fujiwara \cite{BBF-QT} showed that residually finite hyperbolic groups and mapping class groups have ``Property (QT)'': they act properly and isometrically on a finite product of quasi-trees.
In their proof, the quasi-trees are quasi-trees of spaces, built from axes of loxodromic (or hyperbolic) elements in a hyperbolic graph, using projection complexes in the sense of their paper \cite{BBF-IHES}.
In this section, for a group acting on such a graph, we build an affine uniformly Lipschitz action on $\ell^1$ with a certain lower bound on distance. We do this by combining our methods with those from \cite{BBF-QT}, and with a construction of quasi-cocycles on quasi-trees inspired by Bestvina--Bromberg--Fujiwara~\cite{BBF-bdd-cohom-v1}.

\subsection{Quasi-cocycles on quasi-trees}
\label{ssec:quasicocycle-quasitree}
First, we outline a variation on Best\-vina--Bromberg--Fujiwara's quasi-cocycle construction for quasi-trees~\cite{BBF-bdd-cohom-v1}*{\S 5}.  
Their construction is a `quasi-fied' version of that in \S\ref{sec:acylind}, where one replaces $w$ by a high power of an `axial WPD element' $f$.
The main difference between their construction and ours is that $w$ is replaced only by a sufficiently long initial segment of an axis of a hyperbolic element $f$, and the resulting constants are independent of the translation length of $f$.  This is necessary in our later construction of proper cocycles. 
For the case of mapping class groups, we also need to allow for an infinite subgroup $J$ acting trivially on the axis, and we want the group to act `acylindrically modulo $J$'.
Given these changes, and that the preprint \cite{BBF-bdd-cohom-v1} is unpublished (being superseded by~\cite{BBF-bdd-cohom-v2}), we include proofs for what we need.  

Let us describe the setup.

\begin{notation}
Suppose $G$ is a group acting isometrically on a metric space $Y$, and $J \leq G$ is a subgroup.  Let $Y_J = \{y\in Y: jy=y,\, \forall j\in J\}$.
\end{notation}

\begin{definition}
A $C_\gamma$-\emph{quasi-geodesic} in $Y$ is a subset $\gamma$ that is $C_\gamma$-quasi-isometric to $\Z$.
We say that $G$ acts $(D,B)$-\emph{acylindrically on $\gamma$ modulo $J$}, where $J \leq G$ and $D=D(\epsilon), B=B(\epsilon)>0$  are functions in $\epsilon>0$, if $\gamma \subset Y_J$ and for every $\epsilon>0$, if $x,y \in \gamma$ with $d_Y(x,y)\geq D(\epsilon)$ and $x',y'\in Y$ then
\[
	| \{ [g]\in G/J : d_Y(x',gx), d_Y(y',gy) \leq \epsilon \} | \leq B(\epsilon).
\] 
\end{definition} 

We require an extension of Notation~\ref{notat:eps}.
\begin{notation}\label{notat:quasi-trees}
	Let $T$ be a tree.
	For $x,y,p,q \in T$ and $\epsilon, L \geq 0$, we write $[x,y] \subseg_{\epsilon,L} [p,q]$ if $[x,y] \cap B(x,L)$ is contained in the $\epsilon$-neighbourhood of $[p,q]$, if $[x, y]\cap[p, q] \neq \emptyset$ and on the intersection their orientations agree.
\end{notation}

Clearly, if $\epsilon$ is small enough compared to $L$ and the lengths of $[p,q]$ and $[x,y]$, then the condition that $[x, y]\cap[p, q] \neq \emptyset$ becomes superfluous. 

Now let us suppose $Y=Q$ is a $C_Q$-quasi-tree, 
and fix $\phi:Q \to T$ a $C_Q$-quasi-isometry to a tree.

Let $\gamma \subset Q$ be a $C_\gamma$-quasi-geodesic on which $G$ acts $(D,B)$-acylindrically modulo a subgroup $J \leq G$.
Suppose there exists $f \in G$ which acts \emph{hyperbolically} on $Q$, that is, all its orbits are  quasi-geodesics, and $f\gamma=\gamma$, so $\gamma$ is a quasi-axis for $f$.
Suppose $M\geq 1$ is given, and let $F=f^M$.

Fix $x_0 \in \gamma$ and $\epsilon, L\geq 0$, and define $W_+(g)=W_{+,\epsilon, L, F, x_0}(g)$ by
\[
	W_{+}(g) = \{ [h]\in G/J : \exists t\in G \mbox{ s.t. } [\phi(thx_0),\phi(thFx_0)] \subseg_{\epsilon,L} [\phi(tx_0),\phi(tgx_0)] \},
\]	
	and similarly define $W_-(g)=W_{-,\epsilon,L,F,x_0}(g)$ by  
\[
	W_{-}(g) = \{ [h]\in G/J : \exists t\in G  \mbox{ s.t. } [\phi(thx_0),\phi(thFx_0)] \subseg_{\epsilon,L} [\phi(tgx_0),\phi(tx_0)] \}.
\]
These are coarse versions of $w_+(g), w_-(g)$ from \S\ref{sec:acylind}.  The use of an auxiliary tree ensures we get perfect cancelling away from the tripod point in our quasi-cocycle estimate below, but as $\phi$ is likely not equivariant we then also involve $t$ translates to correct this.
Because $x_0, Fx_0 \in Q_J$, $W_{\pm }(g)$ are well defined.  
	Moreover, if $G_{x_0}$ is the stabilizer of $x_0$, then for all $g \in G, g' \in G_{x_0}$ we have that $W_+(gg')=W_+(g)$ and $W_-(gg')=W_-(g)$.

The sets $W_-(g)$ and $W_+(g)$ are essentially independent of the exponent $M$ of $F$, if $M$ is large enough; the main purpose of $F$ is to give an orientation to the quasi-axis $\gamma$.

\begin{theorem}[cf.\ \cite{BBF-bdd-cohom-v1}*{Proposition 5.8}]\label{thm:quasitree-cocycle}
	Let $G$ be a group acting on a $C_Q$-quasi-tree $Q$, $f \in G$ a hyperbolic element with an $f$-invariant $C_\gamma$-quasi-axis $\gamma$ on which $G$ acts $(D,B)$-acylindrically modulo $J \leq G$, and let $x_0 \in \gamma$.

	For every $\epsilon>0$ there exists $L_0=L_0(D,C_Q,C_\gamma,\epsilon)$ such that if $L \geq L_0$ and $M\geq 1$ large enough, for $F=f^M$ the sets $W_{\pm ,\epsilon, L, F, x_0}(g)$ allow to construct quasi-cocycles as follows.

Given an isometric linear action of $G$ on a normed vector space $E$, and a non-zero vector $e\in E$ such that $j\cdot e=e$ for every $j\in J$, the map 
	\[
		\alpha:G\to E, \ \alpha(g) = \sum_{[h] \in W_{+,\epsilon,L,F,x_0}(g)} h \cdot e - \sum_{[h] \in W_{-,\epsilon,L,F,x_0}(g)} h \cdot e 
	\]
	defines a quasi-cocycle 
	with $\Delta(\alpha) \leq C(D,B,C_Q,C_\gamma,L)$,
	and $\alpha(gg')=\alpha(g)$ for all $g \in G, g' \in G_{x_0}$, and with $\alpha(\id)=0$.
\end{theorem}

\begin{remark}
	We emphasise that $L_0$ does not depend on the translation length of $f$ or $F$. 
\end{remark}
\begin{proof}
	{\sc Step 1.} \quad
	To show that $\alpha$ is well-defined, it suffices to check that the sums in $\alpha$ are finite.

	Fix $g\in G$, $\epsilon>0$.
	By hyperbolicity, we have
	\begin{lemma}\label{lem:qtqc1}
		There exists $C_1=C_1(C_Q,C_\gamma, \epsilon)$ so that if $[h] \in W_{+,\epsilon,L,F,x_0}(g)$ for some $L >0$, then for any geodesics $[hx_0,hFx_0], [x_0,gx_0]$, the intersection $[hx_0, hFx_0]\cap B(hx_0, L/C_Q)$ lies in the $C_1$-neighbourhood of $[x_0, gx_0]$.
	\end{lemma}
	Let $D=D(2C_1)$ be given by the acylindricity function for $\gamma$, and suppose $L/C_Q \geq D(2C_1)$, that is $L \geq L_0:= C_Q D(2C_1)$.
	Let $y_0 \in [x_0, Fx_0]$ be a point with $d(x_0,y_0)=L/C_Q$.

	Cover $[x_0, gx_0]$ by finitely many balls of radius $2C_1$; the number of balls is $\leq 1+d(x_0,gx_0)/2C_1$.
	If $[h] \in W_+(g)$, then Lemma~\ref{lem:qtqc1} implies that $hx_0$ and $hy_0$ both lie in one of these finitely many balls.
	Hence it suffices to bound the number of $[h]$ for which $hx_0$ and $hy_0$ map into two specified such balls.
	But as $d(x_0,y_0) \geq D(2C_1)$, by acylindricity, this number is at most $B(2C_1)$.

	The proof that $W_-(g)$ is finite for $L \geq L_0$ is identical.
	\smallskip

	{\noindent \sc Step 2.}\quad We now check that $\alpha$ is a quasi-cocycle.
	For $g,g' \in G$, let us consider
	\begin{equation}\label{eq:qcocycle-check}
		\alpha(gg')-\alpha(g)-g \cdot \alpha(g').
	\end{equation}
	We claim that all but a uniformly bounded number of summands in this expression cancel. 

	Fix geodesics $[x_0,gx_0]$, $[x_0,gg'x_0]$ and $[gx_0,gg'x_0]$ in $Q$, and let $z \in Q$ be a point within $C_Q$ of all three.
	\begin{lemma}
		\label{lem:qtqc2}
		There exists $R_1=R_1(C_Q)$ and a partial pairing up of identical elements in $W_\pm(gg'), W_\pm(g), gW_\pm(g')$ such that the following trichotomies hold:
		\begin{itemize}
			\item Every $[h] \in W_+(gg')$ either pairs with a unique $[h] \in W_+(g)$, or with a unique $[h] \in g W_+(g')$, or $[hx_0,hFx_0]\cap B(hx_0,C_Q L)\cap B(z,R_1) \neq \emptyset$.
			\item	Likewise if $[h] \in W_-(gg')$ then it pairs with a unique $[h] \in W_-(g)$, or with a unique $[h] \in g W_-(g')$, or $[hx_0,hFx_0]\cap B(hx_0,C_Q L)\cap B(z,R_1) \neq \emptyset$.
			\item If $[h] \in W_+(g)$ is not already paired up, then either it pairs with $[h] \in g W_-(g')$ or else $[hx_0,hFx_0]\cap B(z,R_1) \neq \emptyset$.
			\item Finally, if $[h] \in W_-(g)$ is not already paired up, then either it pairs with a unique $[h] \in g W_+(g')$ or else $[hx_0,hFx_0]\cap B(hx_0,C_Q L)\cap B(z,R_1) \neq \emptyset$.
		\end{itemize}
	\end{lemma}
	The key point in the proof of the lemma is the equivariance that is implicit in the definitions of $W_+(g)$ and $W_-(g)$, from which it follows that
	\[
	gW_+(g') = \{ [h]\in G/J : \exists t\in G, [\phi (thx_0),\phi(thFx_0)] \subseg_{\epsilon,L} [\phi(tg x_0),\phi(tgg'x_0)] \},
\]  and a similar expression for $gW_{-}(g')$.  

According to the definition of $\alpha$, any paired terms given by Lemma~\ref{lem:qtqc2} will cancel in \eqref{eq:qcocycle-check}.

	It remains to count the unpaired terms.
	By the argument from Step 1, these unpaired terms have $hx_0$ and $hy_0$ lying within $C_1$ of three geodesic subsegments of $[x_0,gx_0], [x_0,gg'x_0],[gx_0,gg'x_0]$ of length $\leq C_2=C_2(C_Q,L,R_1)$.
	Again, by the argument from Step 1, the number of these terms will be $\leq C$ for $C=C(C_2,B)$.
	Thus $\Delta(\alpha) \leq C \|e\|$ as desired.
	Finally, $\alpha(gg')=\alpha(g)$ for all $g \in G, g' \in G_{x_0}$ by the analogous property for $W_-(g)$ and $W_+(g)$, and $\alpha(\id)=0$ as $W_-(\id)=W_+(\id)=\emptyset$.
\end{proof}

\subsection{Projection complexes and distance estimates}
\label{ssec:proj-complexes-standing-assump}
We now consider how the construction of quasi-cocycles in Theorem~\ref{thm:quasitree-cocycle} behaves when applied to certain quasi-trees that result from projection complexes, with the goal to find lower bounds on their norms.

First, we recall the definition of a quasi-tree of metric spaces resulting from a projection complex.  All our metric spaces are graphs with edge lengths $1$ (or later, $K \geq 1$), but the graphs need not be locally finite.

Suppose $\mathbf{Y}$ is a collection of uniform $C_A$-quasi-geodesics with projections $\{\pi'_Y \mid Y \in\mathbf{Y} \}$. For any $X \in \mathbf{Y}\setminus\{Y\}$, $\pi'_Y(X)$ is a subset of $Y$.
Suppose $(\mathbf{Y},\{\pi'_Y \mid Y \in\mathbf{Y} \})$ satisfy the strong projection axioms with constant $\xi'$, in the sense of \cite{BBFS} and \cite{BBF-QT}*{\S 2.3}.

For $Y \in \mathbf{Y}\setminus \{X,Z\}$, define $d_Y(X,Z) = \diam \left[ \pi'_Y(X)\cup\pi'_Y(Z) \right]$.
For $K\geq 4\xi'$, define a quasi-tree of metric spaces $C_K(\mathbf{Y})$ by taking the disjoint union of spaces in $\mathbf{Y}$ and connecting all pairs of spaces $X,Z\in \mathbf{Y}$ with the property that $d_Y(X,Z)< K$ for all $Y \in \mathbf{Y}\setminus \{X,Z\}$, by joining every point in $\pi'_X(Z)$ with every point in $\pi'_Z(X)$ by an edge of length $K$.
The space $C_K(\mathbf{Y})$ thus obtained is a $C_Q=C_Q(C_A,\xi')$-quasi-tree \cite{BBFS}*{Theorem 6.6}.	 

We extend the projections and distances between them to $C_K(\mathbf{Y})$. If $x\in X \in \mathbf{Y}$ and $Y \in \mathbf{Y} \setminus\{X\}$ set $\pi'_Y(x)=\pi'_Y(X)$, and set $\pi'_X(x)=\{x\}$. For each $Y\in \mathbf{Y}$, extend $d_Y$ to $C_K(\mathbf{Y})$ as follows: suppose $x\in X \in \mathbf{Y}, z\in Z \in \mathbf{Y}$.
 Then if $X\neq Y, Z\neq Y$, set $d_Y(x,z)=d_Y(X,Z)$; if $X \neq Z$ set $d_X(x,z)=\diam \left[ \{x\}\cup\pi'_X(Z) \right]$; and if $X=Z$ set $d_X(x,z)$ to be the given distance in $X$.

We are usually only interested in large values of $d_Y$, therefore we often use the truncated distance $\Tsh{L}{d_Y(\cdot,\cdot)}$, whose meaning is as follows.
\begin{notation}\label{not:truncate-real} For real numbers $L, N>0$, $\Tsh L{N}$
equals $N$ if $N\geq L$, and $0$ otherwise.
\end{notation}
In \cite{BBF-QT}, the authors write $d_Y(\cdot,\cdot)_L$ for $\Tsh{L}{d_Y(\cdot,\cdot)}$.

When constructing proper quasi-cocycles for both residually finite hyperbolic groups and mapping class groups, the following standing assumptions and notation will apply.  Note that given the other assumptions, the final two points are guaranteed to hold for suitable $x_0,\ldots, x_p, \phi$.
\begin{assumptions}\label{assump:quasi-tree-complex}
	Suppose:
	\begin{itemize}
		\item $\cA'$ is a collection of graphs that are $C_A$-quasi-geodesics;
		\item $\{\pi'_Y\}_{Y\in \cA'}$ is a collection of projections so that $(\cA', \{\pi'_Y\})$ satisfy the strong projection axioms with constant $\xi'$; 
		\item $H$ is a finitely generated group acting on the disjoint union $\bigsqcup_{Y\in\cA'}Y$, respecting the partition and not flipping the ends of any $Y$;
		\item  $\cA' = H\cdot \gamma$, for some $\gamma \in \cA'$; 
		\item $J$ is a subgroup which fixes every point of $\gamma$; 
		\item $K \geq 4\xi'$ is given so that $C_K(\cA')$ is a $C_Q$-quasi-tree   ;
		\item $H$ acts on $\gamma\subset C_K(\cA')$ $(D,B)$-acylindrically modulo $J$, and there exists an element $f \in H$ acting hyperbolically, $f\gamma=\gamma$;
		\item $x_0, x_1, \ldots, x_p \in \gamma$ are a collections of basepoints that are moreover representatives of $H$-orbits of vertices in $C_K(\cA')$, so that every point of $C_K(\cA')$ is within distance $C_A$ of such an orbit. Since $\langle f \rangle$ acts coboundedly on the $C_A$--quasi-geodesic $\gamma$, we need only finitely many such orbits. 
	Let $C_F = \max \{ d_\gamma (x_0, x_i) : i=1,\ldots,p\}$.
	  \item $\phi:C_K(\cA')\to T$ is a fixed $C_Q$-quasi-isometry to a tree $T$.  
	\end{itemize}
\end{assumptions}

These assumptions imply that every $Y\in \cA'$ is endowed with an orientation. Indeed, every $h' \in H$ so that $h'\gamma= Y$ induces a well-defined orientation on  $Y$, coming from $f$. As there is no element in the stabilizer of $\gamma$ swapping its endpoints, this orientation is independent of $h'$.

Under these standing assumptions, we are going to use quasi-cocycles of the type constructed in Theorem~\ref{thm:quasitree-cocycle}.

	Recall Notation~\ref{notat:quasi-trees} and the definitions of $W_{\pm, \epsilon, L, F, x_0}(g)$.
If $L$ is large enough then the sets $W_{+,\epsilon,L,F,x_*}(g)$ and $W_{-,\epsilon,L,F,x_*}(g)$ are disjoint, for any choice of basepoint $x_* \in \gamma$:
	\begin{lemma}\label{lem:empty}
		Given $\epsilon\geq 0$, there exists $L'=L'(C_Q, C_A, \epsilon)$ so that for any $L \geq L'$ and any $M$ large enough that $F=f^M$ translates by at least $L$, and any basepoint $x_* \in \gamma$, 
		$W_{+,\epsilon,L,F,x_*}(g) \cap W_{-,\epsilon,L,F,x_*}(g) = \emptyset$.
	\end{lemma}
	\begin{proof}
		Suppose $[h] \in W_{+,\epsilon,L,F,x_*}(g) \cap W_{-,\epsilon,L,F,x_*}(g)$.
		Then, coarsely, the initial segment of $[hx_*, hFx_*]$ travels towards $[x_*, gx_*]$ then follows along $[x_*, gx_*]$ going towards both $x_*$ and $gx_*$, which is absurd for $L$ large enough depending on $\epsilon, C_A, C_Q$.
	\end{proof}

We now relate the sum of the projection distances 
\begin{equation}\label{eq:sum}
\sum_{Y\in \cA'}\Tsh{L}{d_{Y}(x_0, gx_0)}
\end{equation}
to the sizes of the sets $W_{\pm, \epsilon,L,F,x_i}$, $i\in \{ 0,1,\dots ,p\}$.

	\begin{lemma}\label{lem:w-contributions}
		There exists an appropriate choice of constants $\epsilon=\epsilon(C_A,C_Q)$ and $L_0'=L_0'(D,C_Q,C_A,\epsilon)$ to be used in the definition of $W_{\pm, \epsilon,L,F,x_i}$ (with $L_0'$ larger than $L_0(D,C_Q,C_A,\epsilon)$ from Theorem \ref{thm:quasitree-cocycle}, and larger than $L'(C_Q,C_A,\epsilon)$ from Lemma \ref{lem:empty}) and an appropriate choice of the threshold $L_* = L_*(L_0',\epsilon,\allowbreak C_Q, C_A,K)$ to be used in the sum of the projection distances \eqref{eq:sum} so that for every $L \geq L_*$ we can find a bound $C>0$ with the property that every $g\in H$ determines a partition $\cA' = \cA'_1(g) \sqcup \cA'_2(g)$ satisfying the following.
	\begin{itemize}
	\item[(a)] The partial sum of distances corresponding to $\cA'_2(g)$ is bounded:
		\[
			\sum_{Y\in \cA'_2(g)}\Tsh{L}{d_{Y}(x_0, gx_0)} \leq C. 
		\]
	\item[(b)] For each $Y \in \cA'_1(g)$, there exists $i=i_Y \in \{0,1,\ldots,p\}$, $N(Y,g,i) \geq \frac{1}{C}\Tsh{L}{d_Y(x_0,gx_0)}>0$, and $h\in H$ satisfying $h \gamma = Y$ such that:
	
\begin{itemize}
	\item[$\bullet$] $[h], [hf], \ldots, [hf^{N(Y,g,i)}]$
 are distinct elements in $W_{-,\epsilon,L_0',F,x_i}(g) \cup W_{+,\epsilon,L_0',F,x_i}(g)$;
	\item[$\bullet$] for every $[h']\in W_{\pm,\epsilon,L_0',F,x_i}(g)$ with $h' \gamma =Y$ there exists $j \in \{0,\ldots, N(Y,g,i)\}$ such that $d(h'x_i, hf^jx_i) \leq C$.
	\end{itemize}
 	\item[(c)] 	In particular, 
 		\[
 			\sum_{Y\in \cA'_1(g)}\Tsh{L}{d_{Y}(x_0, gx_0)} 
 			\leq C \sum_{i=0}^p \left[ |W_{+,\epsilon,L_0',F,x_i}(g)| +|W_{-,\epsilon,L_0',F,x_i}(g)|\right] .
 				\]
	\end{itemize}
	\end{lemma}
	\begin{proof} There exists $L_1=L_1(C_A, C_Q)$ so that if $L \geq L_1$ then $d_Y(x_0,gx_0) \geq L$ implies that in $C_K(\cA')$ the geodesics from $x_0$ to $gx_0$ pass within a uniformly bounded distance $\delta_Q=\delta_Q(C_Q)$ of both $\pi_Y'(x_0)$ and $\pi_Y'(gx_0)$.
	
\medskip 

		\noindent {\sc{Step 1.}}\quad Let $\cA'_{1a}(g)$ be composed of those $Y \in \cA'$ so that $\pi_Y'(x_0), \pi_Y'(gx_0)$ are both at a distance at least $C_F+\kappa_Q$ from $\{x_0, gx_0\}$, where $\kappa_Q=\kappa_Q(C_Q)$ is large enough compared to $\delta_Q$ and to the bottleneck constant and the hyperbolicity constant of $C_K(\cA')$ so that for every $i$ the projections $\pi_Y'(x_i)$ and $\pi_Y'(gx_i)$ are within distance $2\kappa_Q$ of $\pi_Y'(x_0)$ and $\pi_Y'(gx_0)$, respectively.

Suppose $\pi_Y'(x_0)$ and $\pi_Y'(gx_0)$ appear on $Y$ in that order, with respect to the orientation on $Y$.
		Choose $i \in \{0,\ldots, p\}$ and $h\in H$ with $h\gamma=Y$, so that $hx_i$ is $C_A$-close to $\pi_Y'(x_0)$, and hence $(C_A+2\kappa_Q)$-close to $\pi_Y'(x_i)$.
		(If the order of $\pi_Y'(x_0)$ and $\pi_Y'(gx_0)$ on $Y$ is reversed, we choose $hx_i$ to be $C_A$-close to $\pi_Y'(gx_0)$ and proceed analogously.)

		Let $y_i \in \gamma$ be a point at distance $\in [L,L+C_A]$ from $x_i$ in the direction of $f$, and let $M\geq 1$ be large enough so that for $F=f^M$, $y_i$ is between $x_i$ and $Fx_i$.
		The geodesics in $C_K(\cA')$ from $hx_i$ to $hFx_i$ pass within bounded distance of $hy_i$.  
		Thus there exists $\epsilon=\epsilon( C_A, C_Q)$ so that $\phi(hy_i)$ is within distance $\epsilon$ of $[\phi(hx_i),\phi(hFx_i)]$ in $T$; note that $\epsilon$ does not depend on $L$.
		
		Let the constant $L_0'$ used in the definition of $W_{\pm ,\epsilon,L_0',F,x_i}(g)$ be the maximum of $L_0=L_0(D,C_Q,C_A,\epsilon)$ from Theorem~\ref{thm:quasitree-cocycle} and $L'=L'(C_Q,C_A,\epsilon)$ from Lemma~\ref{lem:empty}.
		Assume moreover that the threshold $L_*=L_*(L_0',\epsilon,C_Q,\allowbreak C_A,K)$ used in the sum of the projection distances \eqref{eq:sum} is the maximum between $L_1(C_A, C_Q)$ mentioned in the beginning of the proof and $C_Q(L_0'+3\epsilon+C_Q) \geq L_0$. Then $d_{C_K(\cA')}(x_i,y_i) \geq L \geq L^*$ implies that in~$T$ $$d_T(\phi(hx_i),\phi(hy_i)) \geq L_0'+3\epsilon.
$$ 

 Moreover, $[\phi(hx_i),\phi(hy_i)]$ and $[\phi(hx_i),\phi(hFx_i)]$ agree in $B(\phi(hx_i),L_0')$.  
		As $\phi(hx_i)$ and $\phi(hy_i)$ are $\geq 3\epsilon$ apart and both are in the $\epsilon$-neighbourhood of $[\phi(x_i), \phi(gx_i)]$, we have that 
		$[\phi(hx_i),\phi(hFx_i)] \subseg_{\epsilon,L_0'} [\phi(x_i),\phi(gx_i)]$, i.e.\ $[h] \in W_{+,\epsilon,L_0',F,x_i}(g)$.  
		Let $\cC(Y,g,i)$ be the set of all $[h'] \in W_{+,\epsilon,L_0',F,x_i}(g) \cup W_{-,\epsilon,L_0',F,x_i}(g)$ with $h'\gamma=Y$. Since all $h\in \cC(Y,g,i)$ are orientation preserving,  $\cC(Y,g,i)$ is entirely contained either in $W_{+,\epsilon,L_0',F,x_i}(g)$ or in $W_{-,\epsilon,L_0',F,x_i}(g)$.
		We have just shown that our chosen $[h]$ is in $\cC(Y,g,i)$ when $d_Y(x_0,gx_0)\geq L$. A similar argument shows that $[hf],[hf^2],\ldots,[hf^{N(Y,g,i)}]$ are in $\cC(Y,g,i)$ for some $N(Y,g,i)\geq\frac{1}{C}\Tsh{L}{d_Y(x_0,gx_0)}$, where $C=C(L,f)\geq 1$ depends on the translation length of $f$.

		Since we can choose $N(Y,g,i)$ so that $hx_i$ and $ hf^{N(Y,g,i)}x_i$ lie close to $\pi_Y'(x_0)$ and $ \pi_Y'(gx_0)$, respectively, we have that  
				every $[h']\in W_{\pm,\epsilon,L_0',F,x_i}(g)$ with $h'\gamma = Y$ has $h'x_i$ in the $C$-neighbourhood of \[ \{hf^jx_i : j=0,\ldots, N(Y,g,i)\}. \]
				\smallskip 

\noindent {\sc{Step 2.}}\quad Let $\cA'_{1b}(g)$ be composed of those $Y \in \cA'$ for which 
		$$d_{C_K(\cA')}( \pi_Y'(x_0) , \pi_Y'(gx_0) ) \geq 5(C_F+C_Q+C_A+\kappa_Q+ L).$$

		Suppose $\pi_Y'(x_0), \pi_Y'(gx_0)$ appear on $Y$ in that order (the other case is similar).  
		Choose points $y,z \in Y$ so that $\pi_Y'(x_0), y, z, \pi_Y'(gx_0)$ appear on $Y$ in that order, and that $d_Y(\pi_Y'(x_0),y), d_Y(z,\pi_Y'(gx_0)) \in [C_F+C_Q,C_F+C_Q+C_A]$.
		Choose $i \in \{0,\ldots,p\}$ and $h\in H$ so that $h\gamma=Y$ and $hx_i$ is $C_A$-close to $y$ and hence $(2C_A+2C_Q + 2C_F + \kappa_Q)$-close  to $\pi_Y'(x_i)$. 

		The same argument as in Step 1 yields a finite sequence 
		$$[h], [hf],[hf^2],\ldots,[hf^{N(Y,g,i)}]$$ contained in $\cC(Y,g,i)$, for some 
		\[ N(Y,g,i)\geq \frac{1}{C}\Tsh{L}{d_Y(y,z)} \geq \frac{1}{2C}\Tsh{L}{d_Y(x_0,gx_0)}, \]
		such that every $[h']\in W_{\pm,\epsilon,L_0',F,x_i}(g)$ has $h'x_i$ in some $C$-neighbourhood of $\{hf^jx_i : j=0,\ldots, N(Y,g,i)\}$, where $C = C(L,f) +C_F+C_Q+C_A$, with $C(L,f)$ the constant appearing in Step 1.

\medskip 

\noindent {\sc{Step 3.}}\quad Let $\cA'_{2a}(g)$ be composed of those $Y \in \cA'$ for which 
		\[ \{ \pi_Y'(x_0), \pi_Y'(gx_0)\} \subset B(x_0, 6(C_F+C_Q+C_A+\kappa_Q+L)). \]
		Assume $L \geq K$.
		If $d_{C_K(\cA')}(x_0,gx_0) \leq 6(C_F+C_Q+C_A+\kappa_Q+L)+\kappa_Q$ then the distance formula \cite{BBFS} gives
		\[
		  	\sum_{Y\in \cA'}\Tsh{L}{d_{Y}(x_0, gx_0)} 
		  	\leq 4 d_{C_K(\cA')}(x_0,gx_0)
		  	\leq C':=28(C_F+C_Q+C_A+\kappa_Q+L),
		  \]
		so the Lemma is proved with $\cA'_2(g)=\cA', \cA'_1(g)=\emptyset$.
		  
		Otherwise, choose $y \in [x_0,gx_0]$ at distance $6(C_F+C_Q+C_A+L)+2\kappa_Q$ from $x_0$, then assuming $L \geq 4\kappa_Q$, for each $Y \in \cA'_{2a}(g)$ we have
		$\Tsh{L}{d_Y(x_0,gx_0)} \leq 2\Tsh{L}{d_Y(x_0,y)}$.
		Thus again by the distance formula,
		\begin{align*}
			\sum_{Y\in \cA'_{2a}(g)} \Tsh{L}{d_Y(x_0,gx_0)}
			& \leq 2 \sum_{Y \in \cA'_{2a}(g)} \Tsh{L}{d_Y(x_0,y)}
			\\ &\leq 8 d_{C_K(\cA')}(x_0,y)
		    \\ & \leq C'':= 48(C_F+C_Q+C_A+L) + 16 \kappa_Q.
		\end{align*}

\medskip 

\noindent {\sc{Step 4.}}\quad Let $\cA'_{2b}(g)$ be composed of those $Y \in \cA'$ for which 
		\[ \{ \pi_Y'(x_0), \pi_Y'(gx_0)\} \subset B(gx_0, 6(C_F+C_Q+C_A+\kappa_Q+L)). \]
		Then as in Step 3, 
		\begin{align*}
			\sum_{Y\in \cA'_{2b}(g)} \Tsh{L}{d_Y(x_0,gx_0)}
		    \leq C''.
		\end{align*}

		\smallskip
		We can now finish the proof of (a) and (b), by taking the constant $C$ to be the maximum of $2C$ and $C'+C''$, and letting $\cA'_1(g) = \cA'_{1a}(g) \cup \cA'_{1b}(g)$ and $\cA'_2(g) = \cA'_{2a}(g)\cup \cA'_{2b}(g)$.

		The statement (c) follows from the fact that, by definition, $Y\neq Y'$ implies $\cC(Y,g,i) \cap \cC(Y',g,i)=\emptyset$. 
	\end{proof}

We will apply these bounds in two different ways in the following two subsections, which are used in the cases of residually finite hyperbolic groups and mapping class groups respectively.
The challenge is to find a combination of quasi-cocycles as in Theorem~\ref{thm:quasitree-cocycle} which is proper, by finding suitable choices of $e$ and using distance formula estimates as in this subsection.

\subsection{A linear lower bound}
Our first (linear) lower bound applies in ``non-elementary'' cases, when the fixed set of $J$ in $C_K(\cA')$ contains points arbitrarily far from $\gamma$.
\begin{proposition}\label{prop:qtqc-lower-bound}
	Suppose that in addition to the Assumptions~\ref{assump:quasi-tree-complex} we have the following property:
	\begin{itemize}
		\item[($\branch$)] for every $i\in \{0,1,\dots ,p\}$ and every $n\in \N$ there exist $h \in H$ with $Jhx_i=hx_i$ and $hx_i$ at distance at least $n$ from $\gamma$.
	\end{itemize}
	Then there exists $L_*=L_*(D,C_A,C_Q, K)$ so that for any $L \geq L_*$ and any $\epsilon>0$ we can find a Banach space $E$ isometrically isomorphic to $\ell^1$, an affine uniformly $(2+\epsilon)$-Lipschitz action of $H$ on $E$, $x \mapsto \pi(g)x+\zeta(g)$, and $C>0$, so that for all $g\in H$,
\begin{equation}\label{eq:lowbound}
		\|\zeta(g)\| \geq \frac{1}{C} \sum_{Y \in \cA'} \Tsh{L}{d_{Y}(x_0,gx_0)} - C.
\end{equation}
\end{proposition}
\begin{proof}
	We use the quasi-cocycles constructed in Theorem~\ref{thm:quasitree-cocycle} combined with the construction described in Proposition \ref{prop:unifLip-action}.

Let $\epsilon=\epsilon(C_A,C_Q)$, $L_0'(D,C_Q,C_A,\epsilon)$ and $L_*(L_0',\epsilon,C_Q,K)$ be the constants given by Lemma~\ref{lem:w-contributions}.

	In order to build our quasi-cocycles, we need the following.
	\begin{lemma}\label{lem:qtqc3}
		For each $i\in \{ 0,\ldots,p\}$, there exist $g_{i,0}, g_{i,1} \in H$ so that $J g_{i,j}x_i = g_{i,j}x_i$, and for every $h\in G$, the triple $\{\phi(hx_i), \phi(hg_{i,0}x_i), \phi(hg_{i,1}x_i)\}$ is not contained in the $\epsilon$-neighbourhood of a geodesic in the tree $T$.
	\end{lemma}
	\begin{proof}
		By our additional assumption ($\branch$), recalling $C_K(\cA')_J$ denotes the set of points fixed by $J$, then $Hx_i \cap C_K(\cA')_J$ contains points arbitrarily far from $\gamma$.
		Choose $g_{i,0} \in H$ so that $g_{i,0} x_i$ lies sufficiently far from $\gamma$ and in $C_K(\cA')_J$, depending on $C_Q,\epsilon$, and then choose $g_{i,1}=f^N$ so that $x_i, g_{i,0}x_i, g_{i,1} x_i$ make a sufficiently non-degenerate tripod in $C_K(\cA')$. The conclusion follows.
	\end{proof}

	For each $i=0,\ldots,p$, consider the action of $H$ on $Hx_i$ by left translation, the associated action on $\ell^1(Hx_i)$ and the coboundary $\beta_i(g)=\delta_{x_i}-\delta_{gx_i}$.
	For $j=0,1$, let $\alpha_{g_{i,j}}$ be the quasi-cocycle given by Theorem~\ref{thm:quasitree-cocycle} for the vector $e_{i,j}=\beta_i(g_{i,j}) \in \ell^1_0 (Hx_i)$ and $L=L_0'$.
	Let $\alpha_i= \alpha_{g_{i,0}}+\alpha_{g_{i,1}} : H \to \ell^1_0(Hx_i)$; as it is a sum of quasi-cocycles, it is also a quasi-cocycle.

	Apply Proposition~\ref{prop:unifLip-action} to find an affine uniformly Lipschitz action of $H$ on the space $V_i=\ell^1_0(Hx_i) \oplus \ell^1_0(Hx_i)$ for which $\zeta_i=(\alpha_i,\beta_i)$ is the cocycle.
	
These actions combine to define an affine uniformly Lipschitz action of $H$ on 
$E=\bigoplus_{i=0}^p V_i$ with the $\ell^1$ metric, for which $(\zeta_i)_{i=0,\ldots,p}$ is the cocycle.

The statements (a) and (c) in Lemma~\ref{lem:w-contributions} imply that 
	\begin{equation*}
		\sum_{Y \in \cA'} \Tsh{L}{d_{Y}(x_0,gx_0)} 
		\leq C + C \sum_{i=0}^p \Big( |W_{-,\epsilon,L_0',F,x_i}(g)| + |W_{+,\epsilon,L_0',F,x_i}(g)| \Big),  
	\end{equation*}
	so it suffices to show that for each $i\in \{ 0,\ldots,p\}$ and $g \in G$:
	\begin{equation}\label{eq:boundwalpha}
		|W_{+,\epsilon,L_0',F,x_i}(g)| +|W_{-,\epsilon,L_0',F,x_i}(g)|
		\leq \|\alpha_{g_{i,0}}(g)+\alpha_{g_{i,1}}(g)\|.
	\end{equation}
	But
	\begin{align*}
		\alpha_{g_{i,0}}(g)+\alpha_{g_{i,1}}(g)
		& = \sum_{[h] \in W_{+,\epsilon,L_0',F,x_i}(g)} \left(2\delta_{hx_i}-\delta_{hg_{i,0}x_i}-\delta_{hg_{i,1}x_i}\right) 
		\\ & \qquad - \sum_{[h] \in W_{-,\epsilon,L_0',F,x_i}(g)} \left(2\delta_{hx_i}-\delta_{hg_{i,0}x_i}-\delta_{hg_{i,1}x_i}\right).
	\end{align*}
	By Lemma~\ref{lem:empty}, $W_+(g)\cap W_-(g) =\emptyset$, and it is impossible for $[h] \in W_+(g)$ and $[h'] \in W_-(g)$ to satisfy $hx_i = h'x_i$.
	Thus, any cancelling of $2\delta_{hx_i}$ for $[h] \in W_+(g)$ can only come from terms of the form $-\delta_{h''g_{i,0}x_i}$ or $-\delta_{h''' g_{i,1}x_i}$ for some other $[h''],[h'''] \in W_+(g)$.
	Lemma~\ref{lem:qtqc3} implies that for each $[h] \in W_+(g)$, at most one of $hg_{i,0}$ and $hg_{i,1}$ could also be in $W_+(g)$.
	Thus
	\[
		\sum_{[h]\in W_{+,\epsilon,L_0',F,x_i}(g)} \left| \left(\alpha_{g_{i,0}}(g)+\alpha_{g_{i,1}}(g)\right)(hx_i) \right| \geq |W_{+,\epsilon,L_0',F,x_i}(g)|.
	\]
	Similarly, considering negative values, we obtain
	\[
		\sum_{[h]\in W_{-,\epsilon,L_0',F,x_i}(g)} \left|\left(\alpha_{g_{i,0}}(g)+\alpha_{g_{i,1}}(g)\right)(hx_i) \right| \geq |W_{-,\epsilon,L_0',F,x_i}(g)|.
	\]
	At this point we have not controlled the Lipschitz constant of the action.  However, as in Proposition~\ref{prop:unifLip-action}, given $\epsilon>0$, up to rescaling each $\alpha_i$ and conjugating the action on each $V_i$ to get an action on $\ell^1$, we can find a uniformly $(2+\epsilon)$-Lipschitz affine action on $\ell^1$ with cocycle comparable to $(\zeta_i)_{i=0,\ldots,p}$. This completes the proof.
\end{proof}

\subsection{A weaker lower bound on distance}
Without the additional assumption ($\branch $), we can only obtain a version of \eqref{eq:lowbound} in which the right hand side is of the form $\theta \left( \frac{1}{C} \sum_{Y \in \cA'} \Tsh{L}{d_{Y}(x_0,gx_0)} - C\right)$ for some proper functions $\theta$. 
The idea is that instead of constructing a quasi-cocycle using a vector $e_{i,0}+e_{i,1}$ with finite support in a tripod shape, we find a vector $e$ with support in $\gamma$ and values that decay at a suitable rate: $e$ is in $\ell^1_0(\gamma)$, but sums of $k$ translations of $e$ have $\ell^1$ norm growing almost linearly in $k$.

We now precisely describe the types of functions $\theta$ that we can thus find.
\begin{assumption}\label{assump:theta}
	Suppose $\theta:\N \to [0,\infty)$ satisfies the following:
	\begin{itemize}
		\item $\theta$ is non-decreasing with $\lim_{t\to\infty} \theta(t) = \infty$;
		\item $\theta$ is sub-additive;
		\item there exists $\Theta: \N \to [0,\infty)$ with $\Theta \in \ell^1(\N)$, and for all $t\in \N$, 
			\[
				\theta(t) \leq \sum_{i=1}^t \sum_{j=i}^\infty \Theta(j).
			\]
	\end{itemize}
\end{assumption}
\begin{lemma}
	For every $k\in \N$, 
	\[ \theta(t) = \frac{t}{1+\log^{\circ k}_+ t}  \]  
	satisfies Assumption~\ref{assump:theta}, where
	$\log^{\circ k}_+(t)$ equals the $k$-fold composition of $\log$ when that is defined and positive, and equals $0$ otherwise.
\end{lemma}
\begin{proof}
	Let 
	\[
		\Theta(t) = \frac{1}{t \log(t) \log(\log(t)) \cdots \log^{\circ (k-1)}(t) (\log^{\circ k}(t))^2},
	\]
	for $t>T$ large enough that the denominator is positive, otherwise set $\Theta(t)=1$.
	Then $\Theta$ is summable by the integral test, repeatedly substituting $t=e^u$, etc. to see that
	\begin{align*}
		\int_{t=T}^\infty  \frac{1}{t \log(t) \cdots (\log^{\circ k}(t))^2} dt
		& = 
		\int_{u=\log(T)}^\infty  \frac{1}{u \log(u) \cdots (\log^{\circ (k-1)}(u))^2} du
		\\ & = \cdots
		= \int_{v=\log^{\circ k}(T)}^\infty  \frac{1}{v^2} dv
		= \frac{1}{\log^{\circ k}(T)} 
		< \infty.
	\end{align*}
	This same calculation shows that for $t>T$,
	\[
		\sum_{i=1}^t \sum_{j=i}^\infty \Theta(j)
		> T+\sum_{i=T}^t \frac{1}{\log^{\circ k}(i)}
		> T+\frac{t-T}{\log^{\circ k}(t)} 
		> \frac{t}{1+\log^{\circ k}(t)}.
	\]
	On the other hand, for $t \leq T$,
	\[
		\sum_{i=1}^t \sum_{j=i}^{\infty} \Theta(j) > \sum_{i=1}^t \Theta(i) = t = \theta(t).
	\]
	For any $s,t \in \N$,
	\[
		\theta(s+t)  
		=  \frac{s}{1+\log^{\circ k}_+ (s+t)} + \frac{t}{1+\log^{\circ k}_+ (s+t)}  
		\leq \theta(s)+\theta(t)
	\]
	since $\log^{\circ k}_+$ is non-decreasing,
	thus $\theta$ is sub-additive.

	Let $T'$ be the $k$-fold composition of $\exp$ applied to $0$, so $\log^{\circ k}(t)$ is well-defined and positive if and only if $t>T'$.
	For $t \leq T'$, $\theta(t)=t$ so is clearly non-decreasing.
	For $t > T'$, 
	\[
		\theta'(t) = \frac{1}{(1+\log^{\circ k}(t))^{2}}
		  \left( 1+ \log^{\circ k}(t) - \frac{t}{ \log^{\circ(k-1)}(t) \cdots \log^{\circ 2}(t) \log(t)t} \right)
		  > 0
	\]
	since for $k =1$, $\log(t)>0$ and 
	for $k \geq 2$ and $t > T'$, 
	\[ 1-\frac{1}{\log^{\circ(k-1)}(t) \cdots \log(t)} > 0. \]
	Thus in all cases $\theta$ is increasing.
\end{proof}

We now prove that, without property $(\branch )$, a weaker lower bound for the norm of the cocycle can be obtained.
\begin{proposition}\label{prop:qtqc-lower-bound-weak}
	Suppose that the Assumptions~\ref{assump:quasi-tree-complex} are satisfied and let $\theta$ be a function satisfying Assumption~\ref{assump:theta}.

	Then there exists $L_*=L_*(D,C_A,C_Q, K)$ so that for every $L \geq L_*$ and every $\epsilon>0$ we can find a Banach space $E$ isometrically isomorphic to $\ell^1$, an affine uniformly $(2+\epsilon)$-Lipschitz action of $H$ on $E$, $x \mapsto \pi(g)x+\zeta(g)$, and $C>0$, so that for all $g\in H$,
	\[
		\|\zeta(g)\| \geq \frac{1}{C}\theta\Bigg( \frac{1}{C}  \sum_{Y \in \cA'} \Tsh{L}{d_{Y}(x_0,gx_0)} -1 \Bigg).
	\]
	
	In the inequality above, we interpret $\theta(t)$ for $t\in \R$ as $\theta(t) = \theta(\max\{0,\lceil t \rceil \})$. 
\end{proposition}

\begin{proof}
For every $i\in \{ 0,1,\dots , p\}$, consider the action of $H$ on $\ell^1 (Hx_i)$ and the coboundary $\beta_i(g)=\delta_{x_i}-\delta_{gx_i}$.

	Consider the constants provided by Lemma \ref{lem:w-contributions}, $\epsilon=\epsilon(C_A,C_Q)$ and $L_0'=L_0'(D,C_Q,C_A,\epsilon)$ to be used in the definition of $W_{\pm, \epsilon,L,F,x_i}(g)$, and the threshold $L_*=L_*(L_0',\epsilon,C_Q, C_A,K)$ to be used in the sum of the projection distances.

For each $i =0, \ldots,p$ we wish to construct a suitable quasi-cocycle using Theorem~\ref{thm:quasitree-cocycle}. In what follows, we use the notation from the conclusion of Lemma \ref{lem:w-contributions}.
 
The main difference is that, in the application of Theorem \ref{thm:quasitree-cocycle}, we cannot choose the vector $e_i=\beta_i(g) \in \ell^1_0(Hx_i)$, or indeed any vector $e_i\in \ell^1_{00}(Hx_i)$, since the subgroup $J$ may, in the absence of the assumption ($\branch$), fix only $\gamma$ and nothing else, so $j\cdot e =e$ would force $gx_i\in \gamma$, and the associated quasi-cocycle would not grow in all directions, e.g. $\|\alpha(f^m)\|$ would remain bounded as $m\to \infty$.  

	Instead, let 
	\[
		e_i = \sum_{j \in \N} \Theta(j) \left(\delta_{f^j x_i} - \delta_{f^{-j}x_i} \right) \in \ell^1_0(Hx_i),
	\] 
	where $\Theta$ is the function from Assumption~\ref{assump:theta} for $\theta$.  The support of $e_i$ lies in $\gamma$ so $j\cdot e_i = e_i$ for all $j\in J$.
	
	Let $\alpha_i:H \to \ell^1_0(Hx_i)$ be the quasi-cocycle given by Theorem~\ref{thm:quasitree-cocycle} for the vector $e_i$. The values of $\alpha_i(g)$ on $Y$ are found by summing $h\cdot e_i$ where $hx_i$ lies along the interval in $Y$ between $\pi'_Y(x_i)$ and $\pi'_Y(gx_i)$, in particular outside a uniformly bounded neighbourhood of this interval the values summed are either all negative or all positive.
	Suppose, without loss of generality, that the orientation on $Y$ is from $\pi'_Y(x_i)$ to $\pi'_Y(gx_i)$, with $hx_i$ close to $\pi'_Y(x_i)$.
	By Assumption~\ref{assump:theta} and the previous discussion, for each $Y$ we can just take the negative values and bound:
	\begin{align*}
		 \sum_{y \in Hx_i \cap Y} |\alpha_i(g)(y)| 
		& \geq \sum_{j<0} \Big| \sum_{[h'] \in \cC(Y,g,i)} (h' \cdot e_i)(hf^jx_i) \Big|
		\\ & \geq \sum_{j<0} \Big| \sum_{k=0}^{N(Y,g,i)} (hf^k \cdot e_i)(hf^j x_i) \Big|
		\\ & \geq \sum_{j<0} \sum_{k=0}^{N(Y,g,i)} \Theta(-j+k)
		\\ & \geq \sum_{k=1}^{N(Y,g,i)}\sum_{j=k}^{\infty} \Theta(j)
		\geq \theta(N(Y,g,i)).
	\end{align*}

	Apply Proposition~\ref{prop:unifLip-action} to find an affine uniformly Lipschitz action of $H$ on the space $V_i = \ell^1_0(Hx_i)\oplus \ell^1_0(Hx_i)$ with cocycle $\zeta_i=(\alpha_i,\beta_i)$.
	Let $E=\bigoplus_{i=0}^p V_i$ with the $\ell^1$-metric, and combine the actions to define an affine uniformly Lipschitz action of $H$ on $E$ with cocycle $\zeta = (\zeta_i)_{i=0,\ldots,p}$.

	We have that 
	\begin{align*}
		\|\zeta(g)\| 
		& = \sum_{i=0}^p \|\zeta_i(g)\|
		\geq \sum_{i=0}^p \sum_{Y \in \cA'} \sum_{y\in Y} |\alpha_i(g)(y)|
		\\ & \geq \sum_{Y\in \cA'_1(g)} \theta(N(Y,g,i_Y))
		 \geq \theta\bigg( \sum_{Y\in \cA'_1(g)} N(Y,g,i_Y) \bigg)
		\\ & \geq  \theta\bigg( 
			\sum_{Y \in \cA'_1(g)} \frac{1}{C} \Tsh{L}{d_Y(x_0,gx_0)}\bigg)
		\\ & \geq \theta\left(  \frac{1}{C} \sum_{Y \in \cA'} 
			\Tsh{L}{d_Y(x_0,gx_0)} -1 \right)
	\end{align*} 
	using the subadditivity of $\theta$ and the bound from Lemma~\ref{lem:w-contributions}, (a).

	Again, at this point we have not controlled the Lipschitz constant of the action.  However, again as in Proposition~\ref{prop:unifLip-action}, given $\epsilon>0$, up to rescaling each $\alpha_i$ and conjugating the action on each $V_i$ to get an action on $\ell^1$, we can find a uniformly $(2+\epsilon)$-Lipschitz affine action on $\ell^1$ with comparable bound on the cocycle.
\end{proof}

\section{Proper quasi-cocycles} 
\label{sec:res-fin-hyp-MCG}
We now combine the methods of the previous section with the results in \cite{BBF-QT}.

\subsection{Residually finite hyperbolic groups}  
\label{ssec:resfinhyp-proper}  

	Bestvina--Bromberg--Fujiwara show the following.
\begin{theorem}[Bestvina--Bromberg--Fujiwara \cite{BBF-QT}]
	\label{thm:bbf-rfh}
	Suppose $G$ is a non-elementary residually finite hyperbolic group with Cayley graph $\Gamma$.
	Then there exist constants $\xi', K, C_Q, C_A $ with the following property:
	For any $L_* \geq K$ one can find:
	\begin{itemize}
		\item a finite index subgroup $H \leq G$, 
		\item finitely many $C_A$-quasi-axes $\gamma_1, \ldots, \gamma_n$ in $\Gamma$ for hyperbolic elements in $G$,
		\item for each $\cA_i=H\cdot \gamma_i$ projections $\{\pi'_Y\}_{Y\in \cA_i}$ satisfying the strong projection axioms with constant $\xi'$, and $C_K(\cA_i)$ a $C_Q$-quasi-tree with $H$ acting $(D,B)$-acylindrically on it modulo the trivial group,
		\item basepoints $x_i\in \gamma_i$, $i=1,\ldots,n$, and a constant $C>0$ so that for any $g \in H$,
			\[
				|g| 
				= d_\Gamma(\id,g)
				\leq C \sum_{i=1}^{n} \sum_{Y\in\cA_i} \Tsh{L_*}{d_{Y}(x_i,gx_i)} +C.
			\]
	\end{itemize}
\end{theorem}
\begin{proof}
	This follows the proof of property (QT) for such $G$ \cite{BBF-QT}*{\S 3} exactly, with the only difference in \cite{BBF-QT}*{\S 3.3} in the application of \cite{BBF-QT}*{Proposition 3.3}: one fixes the segment constant `$L$' to satisfy \cite{BBF-QT}*{Proposition 3.3} for `$K$'$=L_*$.
	The rest of the proof follows to find the finite index subgroup, and uniform quasi-axes, so that the required distance bound holds.
\end{proof}

We now prove:
\begin{varthm}[Theorem \ref{thm:proper-res-fin-hyp}.]
If $G$ is a residually finite hyperbolic group, then, for every $\epsilon>0$, $G$ admits an affine uniformly $(2+\epsilon)$-Lipschitz action on $\ell^1=\ell^1(\N)$ with undistorted orbits, and hence likewise on $L^1=L^1([0,1])$.
\end{varthm}
\begin{proof}
	By Proposition~\ref{prop:lattices}\eqref{item:proper}, we may assume that $G$ is torsion-free. We may also assume that $G$ is non-elementary, since $\Z$ has a standard affine isometric action on $\ell^1(\N)$, generated by $f\mapsto f+\delta_1$.  

	Let $\xi',K,D,B,C_A,C_Q$ be given by Theorem~\ref{thm:bbf-rfh}, let $L_*\geq K$ be then given by Proposition~\ref{prop:qtqc-lower-bound}.
	Apply Theorem~\ref{thm:bbf-rfh} to find the finite index subgroup $H$, collections of axes $\cA_1, \ldots, \cA_n$, and constant $C$ for the distance estimate.

	For $i=1,\ldots,n$ apply Proposition~\ref{prop:qtqc-lower-bound} to find a space $E_i$ isometrically isomorphic to $\ell^1$ with an affine uniformly $(2+\epsilon)$-Lipschitz action on $E_i$ with cocycle $\zeta_i$ so that the cocycle bound is satisfied for $\cA'=\cA_i$ and `$C$' $=C_i$.
	Let $C_*= \max\{C_1,\ldots,C_n\}$.  
	Consider the $\ell^1$ space $\bigoplus_{i=1}^n E_i$ (isometrically isomorphic to $\ell^1$) with the product action having cocycle $(\zeta_i)$.
	Then combining the distance bounds, for any $g\in H$,
	\begin{align*}
		|g| 
		= d_\Gamma(\id,g)
		& \leq C \sum_{i=1}^{n} \sum_{Y\in\cA_i} \Tsh{L_*}{d_{Y}(x_i,g x_i)} +C
		\\ & 
		\leq C C_* \sum_{i=1}^n \|\zeta_i(g)\| +CC_*^2m+C
		\\ & = CC_* \|\zeta(g)\| +CC_*^2m+C. \qedhere
	\end{align*}
\end{proof}

Induction then gives us proper affine uniformly Lipschitz actions on $L^1$ for simple rank $1$ Lie groups.
\begin{proof}[Proof of Corollary~\ref{cor:rank-one-proper}]
	For $G$ a simple Lie group with real rank $1$ and Haar measure $\mu$, let $\Lambda$ be a residually finite uniform lattice in $G$.
	By Theorem~\ref{thm:proper-res-fin-hyp}, $\Lambda$ has a affine uniformly Lipschitz action on $\ell^1(\N)=L^1(\N,\sigma)$ with undistorted orbits, where $\sigma$ is the counting measure.
	By Proposition~\ref{prop:lattices}, (2), $G$ admits a continuous affine uniformly Lipschitz action on $L^1(\Lambda\backslash G \times \N, \mu \times \sigma) \cong L^1([0,1])$ with undistorted orbits.	
\end{proof}

\subsection{Proper quasi-cocycles for mapping class groups}
\label{ssec:mcg}

Let $\Sigma$ be a connected orientable surface of genus $g$ with $p$ punctures and let $\xi (\Sigma )= 3g + p -3$ be the complexity of the surface. In what follows, we consider the mapping class group $\mcgs$ of the surface $\Sigma$, that is, the quotient of the group of homeomorphisms of $\Sigma$ by the subgroup of homeomorphisms isotopic to the identity. The mapping class group of a surface of finite type is finitely generated \cite{Birman:Braids}. For $\xi (\Sigma )\leq 1$, $\mcgs$ is virtually free and therefore has a standard affine isometric action on $\ell^1 (\N )$ such that the orbit map is a quasi-isometric embedding. Hence, in what follows we  assume that $\xi (\Sigma )\geq 2$.   

	Bestvina--Bromberg--Fujiwara show the following \cite{BBF-QT}*{\S 4}.

\begin{theorem}[Bestvina--Bromberg--Fujiwara \cite{BBF-QT}]
	\label{thm:bbf-mcg}
	Suppose $G$ is a finite index subgroup of a mapping class group $\mcgs$ of a surface $\Sigma$ as above.
	Then there exist constants $\xi', K, C_Q, C_A $ and  two functions $D$ and $B$ with the following property:
	For any $L_* \geq K$ one can find:
	\begin{itemize}
		\item a finite index subgroup $H \leq G$, 
		\item a finite collection of subsurfaces $U_i, i\in \{ 1,\ldots, n \},$ and, for each $i$, a $C_A$-quasi-axis $\gamma_i$ in the curve graph $\cC(U_i)$ for a $U_i$-pseudo-Anosov element,
		\item for each $\cA_i=H\cdot \gamma_i$, projections $\{\pi'_Y\}_{Y\in \cA_i}$ satisfying the strong projection axioms with constant $\xi'$, and  $C_K(\cA_i)$ the corresponding $C_Q$-quasi-tree endowed with an $H$-isometric action, as defined in \cite[Section 2.3]{BBF-QT}
		\item the action of $H$ on each $C_K(\cA_i)$, $i\in \{ 1,\ldots, n \},$ is $(D,B)$-acylindrical on $\gamma_i$ modulo $J_i$, where $J_i$ is the subgroup of $H$ that fixes $U_i$, 
		\item basepoints $x_i\in \gamma_i$, $i=1,\ldots,n$, and a constant $C>0$ so that for any $g \in H$,
			\[
				|g| 
				= d_\Gamma(\id,g)
				\leq C \sum_{i=1}^{n} \sum_{Y\in\cA_i} \Tsh{L_*}{d_{Y}(x_i,gx_i)} +C.
			\]
	\end{itemize}
\end{theorem}

\begin{proof}
	This follows the proof of property (QT) for such $G$ \cite{BBF-QT}*{\S 4} nearly exactly, with the only differences:
	\begin{itemize}
		\item when taking the initial finite index subgroup `$G$' in \cite{BBF-QT}*{\S 4.8}, one may intersect with our given $G$ to find a further finite index subgroup.
		\item in \cite{BBF-QT}*{\S 4.6} in the application of \cite{BBF-QT}*{Proposition 4.18}: one chooses `$T$' given a choice of `$K$'$\geq L_*$ so that the distance bound of \cite{BBF-QT}*{Proposition 4.18} holds for `$K$'$=L_*$.
	\end{itemize}
	The rest of the proof follows to find the finite index subgroup, and uniform quasi-axes, so that the required distance bound holds.
\end{proof}

We now show:
\begin{varthm}[Theorem \ref{thm:proper-mcg}.]
The mapping class group, $\mcgs$, of a surface $\Sigma$
	of finite type admits an affine uniformly $(2+\epsilon)$-Lipschitz action on $\ell^1$ (hence also on  $L^1=L^1([0,1])$) with proper orbits, for any $\epsilon>0$.
	
	Moreover, for any function $\theta$ satisfying Assumption~\ref{assump:theta}, one can require that the cocycle $\alpha$ satisfies $\|\alpha(g)\| \geq \frac{1}{C}\theta(|g|)-C$ for some $C$ and all $g\in G$.
\end{varthm}
\begin{proof}
As mentioned in the beginning of the section, without loss of generality we may assume that $\xi (\Sigma ) \geq 2$.

We begin with a choice of finite index subgroup $G\leq \mcgs$ to make sure that there is no axis flipping.

Let constants $\xi',K,C_Q,C_A$ and functions $D,B$ be given by Theorem~\ref{thm:bbf-mcg}. 


Let $L_*\geq K$ be then given by Proposition~\ref{prop:qtqc-lower-bound}.
	Apply Theorem~\ref{thm:bbf-mcg} to find the finite index subgroup $H$, collections of axes $\{\cA_i\}$ and quasi-trees $\{C_K(\cA_i)\}$, and constant $C$ for the distance estimate.
	
	Identically to the proof of Theorem~\ref{thm:proper-res-fin-hyp}, for $i=1,\ldots,n$ apply Proposition~\ref{prop:qtqc-lower-bound-weak} to find a space $E_i$ isometrically isomorphic to $\ell^1$ with an affine uniformly $(2+\epsilon)$-Lipschitz action on $E_i$ with cocycle $\zeta_i$ so that the cocycle bound is satisfied for $\cA'=\cA_i$ and `$C$' $=C_i$.
	Let $C_*= \max\{C_1,\ldots,C_n\}$.  
	Consider the $\ell^1$ space $\bigoplus_{i=1}^n E_i$ with the product action having cocycle $(\zeta_i)$.
	Then combining the distance bounds, for any $g\in H$,
	\begin{align*}
		\|\zeta(g)\|
		& \geq \sum_{i=1}^n \|\zeta_i(g)\|
		\\  & \geq \frac{1}{C_*}\sum_{i=1}^n \theta\left( \frac{1}{C_*} \sum_{Y\in \cA_i} \Tsh{L_*}{d_Y(x_i,gx_i)} -C_*\right) 
		\\ & \geq \frac{1}{C_*}\theta\left(
				\sum_{i=1}^n \frac{1}{C_*} \sum_{Y\in \cA_i} \Tsh{L_*}{d_Y(x_i,gx_i)} -C_*n\right) 
		\\ & \geq \frac{1}{C_*}\theta\left(  \frac{1}{C_*C} \|g\| -1-C_*n \right)
		\\ & \geq \frac{1}{C'} \theta(\|g\|)-C'
	\end{align*}
	by subadditivity, for some suitable $C'$.
\end{proof}

\begin{remark}
The lower bound in Theorem \ref{thm:proper-mcg} can be improved as follows. It is known that $d_\Gamma (\id,g)$ is in fact quasi-isometric to 
$$\sum_{i=1}^{n} \sum_{Y\in\cA_i} \Tsh{L_*}{d_{Y}(x_i,gx_i)}.
$$ 

Assume that $\cA_1, \dots , \cA_k$ correspond to axes of pseudo-Anosovs (which in this case are Dehn twists) on subsurfaces $U_1, \dots , U_k$ that are annuli, that is $\xi (U_i)=-1$, while all the other subsurfaces have $\xi (U_i)>-1$. Each $\cA_j$ satisfies the Assumptions \ref{assump:quasi-tree-complex}, but for $j>k$ the extra assumption ($\branch $) is also satisfied. Therefore, one can find a proper affine uniformly Lipschitz action on $\ell^1$ such that the lower bound of the cocycle is
$$
\theta\left(\frac{1}{C}
				\sum_{i=1}^k  \sum_{Y\in \cA_i} \Tsh{L_*}{d_Y(x_i,gx_i)} -C \right) + \frac{1}{C} \sum_{i=k+1}^n \sum_{Y\in \cA_i} \Tsh{L_*}{d_Y(x_i,gx_i)} -C.  
$$ 

Note that $\sum_{i=k+1}^n  \sum_{Y\in \cA_i} \Tsh{L_*}{d_Y(x_i,gx_i)}$ is quasi-isometric to the Weil--Petersson metric induced on $\mcgs$ \textit{via} its identification with an orbit in the Teichm\"uller space.   
\end{remark}

\end{document}